\documentclass[aos,square]{imsart}

\def\addnotation#1: #2#3{$#1$ \> \parbox{3.8in}{#2 \dotfill \pageref{#3}}\\}
\def\addnotationtwo#1: #2#3#4{$#1$ \> \parbox{3.8in}{#2 \dotfill \pageref{#3}, \pageref{#4}}\\}

\usepackage[OT1]{fontenc}
\usepackage{amsthm,amsmath}
\usepackage[numbers]{natbib}

\usepackage{natbib}
\usepackage[colorlinks,citecolor=blue,urlcolor=blue]{hyperref}
\usepackage{hypernat}
\usepackage{amsthm, amsmath, amssymb, amsfonts, booktabs, graphicx, natbib}

\usepackage{bbm, mathtools}

\usepackage{xparse}
\usepackage{etoolbox} %

\usepackage{versions}
\usepackage{verbatim}

\usepackage{enumitem} %

\usepackage{xr}

\startlocaldefs

\newenvironment{notestoself}{%
  \par \vspace{.2cm} { \bf Note to self:}%
}{%
  \vspace{.2cm}
}

\numberwithin{equation}{section}
\graphicspath{{Figures/}}

\newcommand{\lp}{\left(}
  \newcommand{\rp}{\right)}
\newcommand{\lb}{\left\{}
  \newcommand{\rb}{\right\}}
\newcommand{\ls}{\left[} %
  \newcommand{\rs}{\right]}

\newcommand{\PP}{{\mathbb  P}}
\newcommand{\RR}{{\mathbb  R}}

\newcommand{\ZZ}{{\mathbb Z}}

\DeclareMathOperator*{\dom}{dom}

\DeclareMathOperator{\im}{ran}

\newcommand{\Pdmh}{\mathcal{P}_{M,h}}
\newcommand{\Pmh}{\mathcal{P}_{M,h}}
\newcommand{\Pdms}{\mathcal{P}_{M,s}}

\newcommand{\AL}[2]{I^L_{#1,#2}} %
\newcommand{\AU}[2]{I^U_{#1,#2}}
\newcommand{\IL}[2]{I^L_{#1,#2}}
\newcommand{\IU}[2]{I^U_{#1,#2}}
\renewcommand{\ll}[3]{l_{#1, #2, #3}}
\newcommand{\uu}[3]{u_{#1, #2, #3}}
\newcommand{\NS}[1]{N^S_{#1}}
\newcommand{\NB}[1]{N_{#1}} %
\newcommand{\ibf}{\mathbf{i}}

\newcommand{\abf}{\boldsymbol{\alpha}}
\newcommand{\fuia}{f^U_{\ibf, \abf}}
\newcommand{\flia}{f^L_{\ibf, \abf}}

\newcommand{\vp}{\varphi}

\global\long\def\inv#1{\frac{1}{#1}}

\newtheorem{theorem}{Theorem}[section]
\newtheorem{lemma}{Lemma}[section]
\newtheorem{definition}{Definition}[section]
\newtheorem{corollary}{Corollary}[section]
\newtheorem{proposition}{Proposition}[section]
\newtheorem{example}{Example}[section]
\newtheorem{remark}{Remark}[section]
\newtheorem{assumption}{Assumption}[section]
\newcommand{\prf}{\begin{proof}{Proof}}
  \NewDocumentCommand\Nb{O{}}{
    \ifstrempty{#1}{
      N_{[\, ]}
    }{
      N_{[\, ]}\left( {#1} \right)
    }
  }

  \NewDocumentCommand\cC{O{}}{
    \ifstrempty{#1}{
      \mathcal{C}
    }{
      \mathcal{C} \left( {#1} \right)
    }
  }

  \NewDocumentCommand{\ceil}{s O{} m}{%
    \IfBooleanTF{#1} %
    {\left\lceil#3\right\rceil} %
    {#2\lceil#3#2\rceil} %
  }

  \usepackage[usenames,dvipsnames]{xcolor}

  \endlocaldefs

  \excludeversion{mynotestoself}
  \excludeversion{oldtext}

\begin{document}

  \begin{frontmatter}
    \title{Global Rates of Convergence of the MLEs of  Log-concave and $s-$concave Densities}
    \runtitle{Log- and S-Concave MLE Global Rates}

    \begin{aug}
      \author{\fnms{Charles R.}
        \snm{Doss}\thanksref{t1,m1}\ead[label=e1]{cdoss@stat.umn.edu}\ead[label=u2,url]{http://users.stat.umn.edu/\textasciitilde
          cdoss/}}
      \and
      \author{\fnms{Jon A.} \snm{Wellner}\thanksref{t2,m1}\ead[label=e2]{jaw@stat.washington.edu}}
      \ead[label=u1,url]{http://www.stat.washington.edu/jaw/}

      \thankstext{t1}{Supported by NSF Grant DMS-1104832}
      \thankstext{t2}{Supported in part by NSF Grant DMS-1104832 and NI-AID grant 2R01 AI291968-04}
      \runauthor{Doss and Wellner}

      \affiliation{University of Minnesota; \ University of Washington\thanksmark{m1}}

      \address{School of Statistics \\University of Minnesota\\Minneapolis,  MN 55455\\
        \printead{e1}\\
        \printead{u2}}

      \address{Department of Statistics, Box 354322\\University of Washington\\Seattle, WA  98195-4322\\
        \printead{e2}\\
        \printead{u1}}
    \end{aug}

    \begin{abstract}
      We establish global rates of convergence for the Maximum Likelihood
      Estimators (MLEs) of log-concave and $s$-concave densities on $\RR$.
      The main finding is that the rate of convergence of the MLE in the
      Hellinger metric is no worse than $n^{-2/5}$ when $-1 < s < \infty$
      where $s=0$ corresponds to the log-concave case.  We also show that the
      MLE does not exist for the classes of $s$-concave densities with $s < -
      1$.
    \end{abstract}



  \end{frontmatter}

  \smallskip

  \section{Introduction and overview}
  \label{sec:intro}

\subsection{Preliminary definitions and notation}
  We study global rates of convergence of nonparametric estimators of
  log-concave and $s$-concave densities, with focus on maximum likelihood
  estimation and the Hellinger metric.  A density $p$ on $\RR^d$ is {\sl log-concave} if
  \begin{eqnarray*}
    p = e^{\varphi} \ \ \mbox{where} \ \ \ \varphi : \RR^d \mapsto [-\infty,
    \infty)  \ \ \mbox{is concave} .
  \end{eqnarray*}
  We denote the class of all such densities $p$ on $\RR^d$ by ${\cal P}_{d,0}$.
  Log-concave densities are always unimodal and have convex level sets.
  Furthermore, log-concavity is preserved under marginalization and convolution.
  Thus the classes of log-concave densities can be viewed
  as natural nonparametric extensions of the class of Gaussian densities.

  The classes of log-concave densities on $\RR $ and $\RR^d$ are special cases of the classes of
  $s-$concave densities studied and developed by
  \cite{MR0388475,MR0404559},   %
  \cite{MR0450480},  %
  and \cite{MR0428540}.   %
  \cite{MR954608}, pages 84-99, gives a useful summary.
  These classes are defined by the generalized means of order $s$ as follows.
  Let
  \begin{eqnarray*}
    M_s (a,b; \theta) \equiv \left \{
      \begin{array}{l l} ( (1-\theta) a^s + \theta b^s )^{1/s}, & s \not= 0, \ a,b \ge 0, \\
        a^{1-\theta} b^{\theta}, & s = 0, \\   \min ( a,b ) , & s = -\infty .
      \end{array} \right .
  \end{eqnarray*}
  Then $p \in \widetilde{{\cal P}}_{d,s}$, the class of $s-$concave densities
  on $C \subset \RR^d$, if $p$ satisfies
  $$
  p((1-\theta)x_0 + \theta x_1) \ge M_s ( p(x_0) , p(x_1); \theta )
  $$
  for all $x_0, x_1 \in C$ and $\theta \in (0,1)$.  It is not hard to see
  that $\widetilde{{\cal P}}_{d,0} = {\cal P}_{d,0}$ consists of densities of the form $p =
  e^{\varphi}$ where $\varphi \in [-\infty, \infty)$ is concave;
  densities $p$ in $\widetilde{{\cal P}}_{d,s}$ with $s<0$ have the form
  $p = \varphi^{1/s}$ where $\varphi \in [0,\infty)$ is convex; and densities $p$ with $s>0$
  have the form $p= \varphi_{+}^{1/s}$ where $x_+ = \max(x,0)$
  and $\varphi$ is concave on $C$ (and
  then we write $\widetilde{{\cal P}}_{d,s} (C)$); see for example
  \cite{MR954608} page 86.
  These classes are nested since
  \begin{align}
    \label{sConcaveNestingProperty}
    \widetilde{ {\cal P}}_{d,s}(C) \subset \widetilde{{\cal P}}_{d,0} 
    \subset \widetilde{{\cal P}}_{d,r} \subset \widetilde{{\cal P}}_{d,-\infty}  ,
     \quad \mbox{if}  \ -\infty < r < 0< s < \infty.
  \end{align}
  Here we view the classes $\widetilde{\cal P}_{1,s}$ defined above for $d=1$ in terms of the generalized means $M_s$ as being
  obtained as increasing transforms  $h_s$ of the class of concave functions on $\RR$ with
  \begin{eqnarray*}
    h_s (y) = \left \{ \begin{array}{l l}  e^y , & s = 0, \\ (-y)_+^{1/s}, & s < 0 , \\ y_+^{1/s}, & s > 0 .
      \end{array} \right .
  \end{eqnarray*}
  Thus with $\lambda$ denoting Lebesgue measure on $\RR^d$ we define
  \begin{equation*}
    \mathcal{P}_{d,s} = \lb p = h_s(\vp) \colon \vp \mbox{ is concave on } \RR^d \rb
    \bigcap \lb p \colon \int p \, d\lambda = 1 \rb
  \end{equation*}
  where the concave functions $\vp$ are assumed to be closed (i.e. upper
  semicontinuous), proper, and are viewed as concave functions on all of
  $\RR^d$ rather than on a (possibly) specific convex set $C$.
  Thus we consider $\varphi$ as a function from $\RR$ into
  $ [-\infty, \infty)$.  See (\ref{eq:defn:concavefunction})
  in Section~\ref{sec:BasicDefnsAndNotation}.  This view simplifies our
  treatment in much the same way as the treatment in \cite{MR2766867}, but
  with ``increasing'' transformations replacing the ``decreasing''
  transformations of Seregin and Wellner, and ``concave functions'' here
  replacing the ``convex functions'' of Seregin and Wellner.

\subsection{Motivations and rationale}
There are many reasons to consider the $s-$concave classes ${\cal P}_s $ with
$s \ne 0$, and especially those with $s<0$. In particular, these classes
contain the log-concave densities corresponding to $s=0$, while retaining the
desirable feature of being unimodal (or quasi-concave), and allowing many
densities with tails heavier than the exponential tails characteristic of the
log-concave class.  In particular the classes ${\cal P}_{1,s} $ with $s \le
-1/2$ contain all the $t_{\nu}-$ densities with degrees of freedom $\nu \ge
1$.  Thus choice of an $s-$concave class ${\cal P}_s$ may be viewed as a
choice of how far to go in including heavy tailed densities.  For example,
choosing $s=1/2$ yields a class which includes all the $t_{\nu}-$densities
with $\nu \ge 1$ (and all the classes ${\cal P}_s$ with $s>-1/2$ since the
classes are nested), but not the $t_{\nu}-$densities for any $\nu \in (0,1)$.
Once a class ${\cal P}_s$ is fixed, it is known that the MLE over ${\cal
  P}_s$ exists (for sufficiently large sample size $n$) without any choice of
tuning parameters, and, as will be reviewed in
Theorem~\ref{thm:ConsistencySummary-all},
below, is consistent in several senses.
The choice of $s$ plays a role somewhat analogous to some index of smoothness,
$\alpha$ say, in more classical nonparametric
estimation based on smoothness assumptions:
smaller values of $s$ yield larger classes of densities, much as smaller values
of a smoothness index $\alpha$ yield larger classes of densities.
But for the shape constrained families ${\cal P}_s$, no bandwidth or
other tuning parameter is needed to define the estimator, whereas such tuning parameters
are typically needed for estimation in classes defined by smoothness conditions.
For further examples and motivations for the classes ${\cal P}_s$, see  \cite{MR0404559}
and \cite{MR1375234}.
  %
 %
%
%
%
%
%
%
%
 Heavy tailed data are quite common in many application areas including
  data arising from financial instruments (such as
  stock returns, commodity returns, and currency exchange rates),
  and measurements that arise from data
  networks (such as sizes of files being transmitted, file transmission
  rates, and durations of file transmissions) often empirically exhibit heavy
  tails.  Yet another setting where heavy-tailed data arise is in the
  purchasing of reinsurance: small insurance companies may themselves buy
  insurance from a larger company to cover possible extreme losses.  Assuming
  such losses to be heavy-tailed is natural since they are by definition
  extreme.
Two references
(of many) %
 providing discussion of these examples and of inference in
heavy-tailed settings are \cite{MR1652283} %
and \cite{MR2271424}. %

\subsection{Review of progress on the statistical side}

  Nonparametric estimation of log-concave and
  $s$-concave densities has developed rapidly in the last decade.  Here is a brief review
  of recent progress.
\subsubsection{Log-concave and $d=1$}
  For log-concave densities on $\RR$, \\
  \cite{MR2459192} %
  established existence of the Maximum Likelihood Estimator (MLE)
  $\widehat{p}_n$ of $p_0$, provided a method to compute it, and showed that
  it is Hellinger consistent: $H(\widehat{p}_n , p_0 ) \rightarrow_{a.s.} 0$
  where $H^2 (p,q) = (1/2) \int \{ \sqrt{p} - \sqrt{q} \}^2 dx$ is the
  (squared) Hellinger distance.
  \cite{MR2546798} %
  also discussed algorithms to compute $\widehat{p}_n$ and rates of
  convergence with respect to supremum metrics on compact subsets of the
  support of $p_0$ under H\"older smoothness assumptions on $p_0$.
  \cite{MR2509075} %
  established limit distribution theory for the MLE of a log-concave density
  at fixed points under various differentiability assumptions and
  investigated the natural mode estimator associated with the MLE.
\subsubsection{Log-concave and $d\ge 2$}
 Estimation of log-concave densities on $\RR^d$ with $d\ge2$ was initiated
 by \cite{MR2758237};   %
 they established existence and uniqueness and algorithms for computation.
\cite{MR2645484} %
proved consistency in weighted $L_1$ and appropriate supremum metrics, while
\cite{MR2816336,MR2838319}   %
investigated stability and robustness properties and use of the log-concave MLE
in regression problems.
Recently
\cite{Kim:2014wa}  %
 study upper and lower bounds for minimax risks based on Hellinger loss.
When specialized to $d=1$ and $s=0$ their results are consistent with (and somewhat
stronger than) the results we obtain here.
(See Section~\ref{sec:SummaryProblemsProspects} for further discussion.)

\subsubsection{$s-$concave and $d\ge 1$}
While the log-concave (or $0$-concave) case has received the most attention among the
$s$-concave classes, some progress has been made for other $s$-concave classes.
 \cite{MR2766867} showed that the MLE exists
  and is Hellinger consistent for the classes ${\cal P}_{d,s}$ with $s \in (-1/d, \infty)$.
  \cite{MR2722462} studied estimation over $s$-concave
  classes via estimators based on R\'enyi and other divergence criteria rather than
  maximum likelihood.   Consistency and stability results for these divergence
  estimator analogous  to those
  established by \cite{MR2816336} and \cite{MR2838319} for the MLE in the log-concave case have been
  investigated by \cite{Han-Wellner:2015}.

\subsection{What we do here}

  In this paper, we will focus on global rates of convergence of MLE's for the case $d=1$.
  We make this choice because of additional
  technical difficulties when $d>1$.  Although it has been conjectured
  that the $s$-concave MLE is Hellinger consistent at rate $n^{-2/5}$ in the
  one-dimensional cases (see e.g.\ \cite{MR2766867}, %
  pages 3778-3779),   \label{query:pagenumbers}  to the
  best of our knowledge this has not yet been proved (even though it follows for $s=0$
  and $d=1$ from the unpublished results of
  \cite{Doss-Wellner:2013} and \cite{Kim:2014wa}).

 The main difficulty in
  establishing global rates of convergence with respect to the Hellinger or
  other metrics has been to derive suitable bounds for the metric entropy
  with bracketing for appropriately large subclasses ${\cal P}$ of
  log-concave or $s$-concave densities.  We obtain bounds of the form
  \begin{eqnarray}
    \log N_{[\,]} (\epsilon , {\cal P} , H )
    \le K  \epsilon^{-1/2} ,  \ \ \epsilon > 0
    \label{BracketingEntropyWRTHellingerBound1}
  \end{eqnarray}
  where $N_{[\, ]} (\epsilon, {\cal P}, H) $ denotes the minimal number of
  $\epsilon-$brackets with respect to the Hellinger metric $H$ needed to
  cover ${\cal P}$.  We will establish such bounds in
  Section~\ref{sec:BracketEntropyBounds}
  using recent results of
  \cite{MR2519658} %
  (see also \cite{Gunt-Sen:12}) %
  for convex functions on $\RR$. %
  These recent results build on earlier work by \cite{MR0415155} %
  and \cite{MR876079}; see also \cite{MR1720712}, %
  pages 269-281.  The main difficulty has been that the bounds of
  \cite{MR0415155} involve restrictions on the Lipschitz behavior of the
  convex functions involved as well as bounds on the supremum norm of the
  functions.  The classes of log-concave functions to be considered must
  include the estimators $\widehat{p}_n$ (at least with arbitrarily high
  probability for large $n$).
  Since the
  estimators $\widehat{p}_n$ are discontinuous at
  the boundary of their support
  (which is contained in the support of the true density
  $p_0$), the supremum norm does not give control of the Lipschitz behavior
  of the estimators in neighborhoods of the %
  boundary  %
  of their support.
  \cite{MR2519658} showed how to get rid of the constraint on Lipschitz
  behavior when moving from metric entropy with respect to supremum norms to
  metric entropies with respect to $L_r$ norms.  Furthermore,
  \cite{Gunt-Sen:12} showed how to extend Dryanov's results from $\RR$ to
  $\RR^d$ and the particular domains $[0,1]^d$.
  Here we show how the results of \cite{MR2519658} and
  \cite{Gunt-Sen:12} can be strengthened from metric entropy with respect to
  $L_r$ to bracketing entropy with respect to $L_r$, and we carry these
  results over to the class of concave-transformed densities.
  Once bounds of the form (\ref{BracketingEntropyWRTHellingerBound1}) are available,
  then %
   tools from empirical process theory due %
   to
  \cite{MR1240719}, %
  \cite{MR1212164}, %
  \cite{MR1332570},  %
  and developed further in
  \cite{MR1739079} and \cite{MR1385671}, become available.

  The major results in this paper are developed for classes of densities,
  more general than the $s$-concave classes, which we call {\sl concave-transformed
  classes}.  (They will be rigorously defined later, see Section~\ref{sec:Appendix0}.)
  These are the classes studied in \cite{MR2766867}. %
  The main reason for this generality is that it does not complicate the
  proofs, and, in fact, actually makes the proofs easier to understand.  For
  instance, when $h(y) = e^y$, $h'(y)=h(y)$, but the proofs are more
  intuitively understood if one can tell the difference between
  $h'$ and $h$.  %
  Similarly, this generality allows us to keep track of the tail behavior and
  the peak behavior of the concave-transformed classes separately (via the
  parameters $\alpha$ and $\beta$, see
  page~\pageref{assum:transformation}). The tail behavior turns out to be
  relevant for global rates of convergence, as we see in this paper.

Here is an outline of the rest of our paper.
In Section~\ref{sec:MLEbasics} we define the MLE's for $s-$concave classes and briefly
review known properties of these estimators.  We also show that the MLE does not
exist for ${\cal P}_s$ for any $s< -1$.
In Section~\ref{sec:BracketEntropyBounds} we state our main rate
results for the MLE's over the classes ${\cal P}_{1,s}$ with $s>-1$.
In Section~\ref{sec:Appendix0} we state our main general rate results for $h-$transformed
concave classes.
Section~\ref{sec:SummaryProblemsProspects} gives a summary as well as further problems and prospects.
The proofs are given in Sections~\ref{sec:mainresults-proofs}  and ~\ref{sec:Appendix1}.

  \section{Maximum likelihood estimators:  basic properties}
  \label{sec:MLEbasics}
  \label{sec:BasicDefnsAndNotation} %
  We will restrict attention to the class of concave functions
  \begin{equation}
    \label{eq:defn:concavefunction}
    \mathcal{C}
    := \left\{
      \varphi: \RR \to [-\infty, \infty)
      \vert
      \varphi  \mbox{ is a closed, proper concave function}
    \right\},
  \end{equation}
  where \cite{MR0274683} %
  defines proper (page 24) and closed (page 52) convex functions.  A concave
  function is proper or closed if its negative is a proper or closed convex
  function, respectively.  Since we are focusing on the case $d=1$, we
  write ${\cal P}_s$ for ${\cal P}_{1,s}$; this can be written as
  \begin{equation}
    \label{eq:2}
    \mathcal{P}_{s} =  \lb p : \int p \,
    d\lambda =1 \rb \ \ \bigcap \ \ h_s \circ \mathcal{C}
  \end{equation}
  We also follow the convention that all concave functions $\varphi$ are
  defined on all of $\RR$ and take the value $-\infty$ off of their {\em
    effective domains}, $\dom \varphi := \{x : \varphi(x) > -\infty\}$.
  These conventions are motivated in \cite{MR0274683} (page
  40). %
  For any unimodal function $p$, we let $m_p$ denote the (smallest) mode of
  $p$.  %
  For two functions $f$ and $g$ and $r \ge 1$, we let $L_r(f,g) = \Vert f-g
  \Vert_r = \lp \int | f - g |^r d\lambda \rp^{1/r}.$
  We will make the following assumption.
  \begin{assumption}
    \label{assump:TrueDensityHypothesis-s}
    We assume that $X_{i}$, $i=1,\ldots, n$ are i.i.d.\ random variables
    having density $p_0 = h_s \circ \varphi_0 \in {\cal P}_{s}$ for $s \in
    \RR$.
  \end{assumption}
  Write $\PP_n = n^{-1}
  \sum_{i=1}^n \delta_{X_i}$ for the empirical measure of the $X_i$'s.  The
  maximum likelihood estimator $\widehat{p}_n = h_s(\widehat{\vp}_n)$ of
  $p_0$ maximizes
  \begin{equation*}
    \Psi_{n}(\vp) = \PP_n \log p = \PP_n (\log h_s) \circ \vp     %
  \end{equation*}
  over all
  functions $\vp \in \mathcal{C}$ for which $\int h_s(\vp) d\lambda = 1$.
  When $s > -1$, from \cite{MR2766867} %
  (Theorem 2.12, page 3757) we know that $\widehat{\varphi}_n$ exists if
  $n\ge \gamma/(\gamma-1)$
  with
  $\gamma \equiv -1/s> 1$ %
  in the case $s<0$,
  and if
  $n\ge 2$
  when $s\ge 0$.
  \cite{MR2766867}, page 3762,
  conjectured that $\widehat{\varphi}_n$ is unique when it exists.
  See also
  \cite{MR1941467},   %
  \cite{MR2459192} %
  and \cite{MR2546798} %
  (Theorem 2.1) for the $s=0$ case.

  The existence of the MLE has been shown only when $s > -1$.
  One might wonder if this is a deficiency in the proofs or is fundamental.  It is
  well-known that the MLE does not exist for the class of unimodal densities,
  $\mathcal{P}_{-\infty}$; see for example \cite{MR1447736}. The following
  proposition shows that in fact the MLE does not exist for $\mathcal{P}_s$
  when $s < -1$.  The case $ s = -1$ is still not resolved.

  \begin{proposition}
    \label{prop:MLEnonexistence}
    A maximum likelihood estimator does not exist
    for the class ${\cal P}_{s}$ for any $s < -1$.
  \end{proposition}

  Proposition~\ref{prop:MLEnonexistence} gives a negative result about the MLE for an
  $s$-concave density when $s < -1$.  When $s > -1$, there are many known
  positive results, some of which are summarized in the next theorem, which
  gives boundedness and consistency results.  In particular, we already know
  that the MLEs for $s$-concave densities are Hellinger consistent; our main
  Theorem~\ref{thm:HellingerConvergenceRate} extends this result to give the
  rate of convergence, when $s > -1$.

  Additionally, from lemmas and corollaries involved in the proof of
  Hellinger consistency, we know that on compact sets strictly contained in the
  support of $p_0$ we have uniform convergence, and we know that the
  $s$-concave MLE is uniformly bounded almost surely.  We will need these
  latter two results in the proof of the rate theorem to show we only need to
  control the bracketing entropy of an appropriate subclass of
  $\mathcal{P}_{s}$.
  \begin{theorem}[Consistency and boundedness of $\ \widehat p_n$ for $\mathcal{P}_{s}$]
    \label{thm:ConsistencySummary-all}
    Let Assumption~\ref{assump:TrueDensityHypothesis-s} hold with $s > -1$
    and let $\widehat p_n$ be the corresponding MLE.  Then
    \begin{enumerate}[label=(\roman*)]%
    \item \label{enum:Hellinger-consistency}
      $H(\widehat{p}_n , p_0 ) \rightarrow_{a.s.} 0$ as $n \to \infty$,
    \item \label{enum:compacta-consistency}
      If $S$ is a compact set strictly contained in the support of $p_0$,
      $$
      \sup_{x \in S} | \widehat{p}_n (x) - p_0 (x) | \rightarrow_{a.s.} 0
      \mbox { as } n \to \infty,
      $$
    \item \label{enum:a.s.-boundedness}
      $\limsup_{n\rightarrow \infty} \sup_x \widehat{p}_n (x) \le
      \sup_x p_0(x) \equiv M_0 
      < \infty$ almost surely.
    \end{enumerate}
  \end{theorem}
  \begin{proof}

    The first statement \ref{enum:Hellinger-consistency} is proved by
    \cite{MR2459192} for $s=0$, and for
    $s> -1$
    in Theorem 2.17 of \cite{MR2766867}. %
    Statement~\ref{enum:compacta-consistency} for $s=0$ is a
    corollary of Theorem 4.1 of \cite{MR2546798}, and for %
    $s > -1$ follows from Theorem 2.18 of \cite{MR2766867}. %
    Statement~\ref{enum:a.s.-boundedness} is Theorem 3.2 of
    \cite{MR2459192} for $s=0,$ %
    and is Lemma 3.17 in \cite{MR2766867}
    for %
    $s > -1$.
  \end{proof}

  In order to find the  Hellinger rate of convergence of the MLEs, we will bound the
  bracketing entropy of classes of $s$-concave densities.
  In general, by using known consistency results, one does not need to bound
  the bracketing entropy of the entire function class being considered, but
  rather of a smaller subclass in which the MLE is known to lie with high
  probability.  This is the approach we will take, by using parts
  \ref{enum:compacta-consistency} and \ref{enum:a.s.-boundedness} of
  Theorem~\ref{thm:ConsistencySummary-all}.  %
  We therefore consider the following subclasses %
  $\Pdms$ of $s$-concave densities which (we show in the proof of
  Theorem~\ref{thm:HellingerConvergenceRate}) for some $M < \infty$ will
  contain both $p_0$ and
 $\widehat{p}_n$, after translation and rescaling,
  with high probability for large $n$.
 (Recall, the Hellinger distance is invariant
  under translations and rescalings.)  For $0 < M < \infty$, let
  \begin{equation}
    \label{eq:DefnOfCalPMOne}
    \Pdms \equiv
    \lb p \in \mathcal{P}_{s} \colon
    \sup_{x \in \RR} p(x) \le M, 1/M \le p(x) \mbox{ for all } \vert x \vert \le 1 \rb.
  \end{equation}
  The next proposition gives an envelope for the class $\Pdms$.  This
  envelope is an important part of the proof of the bracketing entropy of the
  class $\Pdms$.
  \begin{proposition}
    \label{prop:calPMdSEnvelope}
    Fix $0 < M < \infty$ and  $s > -1$.  Then there exists a constant
    $0 < L < \infty$ depending only on $s$ and $M$ such that  for any $p \in \Pdms$
    \begin{eqnarray}
      p(x)
      & \le & \left \{ \begin{array}{l l}
          \left ( M^s  +  \frac{L}{2M} |x|  \right )^{1/s}, & |x| \ge 2M+1, \\
          M , & |x| < 2M+1. %
        \end{array} \right \}       \label{eq:DefnOfCalPMdSEnvelope}
    \end{eqnarray}
  \end{proposition}
  \begin{proof}
    A corresponding statement for the more general $h$-transformed density
    classes is given in Proposition~\ref{prop:calPMOnehEnvelope} in the
    appendix. However, \eqref{eq:DefnOfCalPMdSEnvelope} does not immediately
    follow from the statement of Proposition~\ref{prop:calPMOnehEnvelope}
    applied to $h \equiv h_s(y) = ( -y)_+^{1/s}$,
    since the requirement $ \alpha > -1/s$ disallows the case $\alpha =
    -1/s$, which is what we need.  However, \eqref{eq:h-simple}
    from the proof of
    Proposition~\ref{prop:calPMOnehEnvelope}
    with $h^{-1}_s(y) = -y^s $ for $y \in
    (0, \infty)$, yields
    \begin{equation*}
      p(x) \le  h_s\lp -M^s - \frac{L}{2M} |x| \rp
    \end{equation*}
    for $|x| \ge 2M+1$, which gives us \eqref{eq:DefnOfCalPMdSEnvelope}.
  \end{proof}

  \section{Main Results: log-concave and $s$-concave classes}
  \label{sec:BracketEntropyBounds}

  Our main goal is to establish rates of convergence for the Hellinger
  consistency given in \ref{enum:Hellinger-consistency} of
  Theorem~\ref{thm:ConsistencySummary-all} for the $s$-concave MLE.
  As mentioned earlier, the key step towards proving rate results of this
  type is to bound the size, in terms of bracketing entropy, of the function
  class over which we are estimating.  Thus we have two main results in this
  section.  In the first we bound the bracketing entropy of certain
  $s$-concave classes of functions.  This shows that for appropriate values
  of $s$, the transformed classes have the same relevant metric structure as
  (compact) classes of concave functions.  Next, using the bracketing bound,
  our next main result gives the rates of convergence of the $s$-concave
  MLEs.

  Now let the bracketing entropy of a class of functions ${\cal F}$ with
  respect to a semi-metric $d$ on ${\cal F}$ be defined in the usual way; see
  e.g.  \cite{MR1720712} page 234, %
  \cite{MR1385671}, page 83, %
  or \cite{MR1739079}, page 16.  %
  The $L_r$-size of the brackets depends on the relationship of $s$ and $r.$
  In particular, for our results, we need to have light enough tails,
  which is to say we need $-1/s$ to be large enough.
  Our main results are as follows:
  \begin{theorem}
    \label{thm:BracketEntropyBound}
    Let $r \ge 1$ and $M > 0$.  Assume that either $s \ge 0 $ or that
    $\gamma \equiv -1/s > 2/r$.
    Then
    \begin{equation}
      \label{eq:bracketF}
      \log N_{[\, ]}(\epsilon, \mathcal{P}_{M, s}^{1/2}, L_r)
      \lesssim \epsilon^{- 1/2},
    \end{equation}
    where the constants in $\lesssim$ depend only on $r$, $M$, and $s$.
    By taking $r=2$ and $s > -1$ we have that
    \begin{equation*}
      \log N_{[\, ]}(\epsilon, \mathcal{P}_{M, s}, H)
      \lesssim
      \epsilon^{-1/2}.
    \end{equation*}
  \end{theorem}
  \smallskip

  Theorem~\ref{thm:BracketEntropyBound}
  is the main tool we need to obtain rates of convergence for the MLEs $\widehat{p}_n$.
  This is given in our second main theorem:
  \begin{theorem}
    \label{thm:HellingerConvergenceRate}
    Let Assumption \ref{assump:TrueDensityHypothesis-s} hold, and let $s >
    -1$. %
    Suppose that $\widehat{p}_{n,s}$ is the MLE of the $s$-concave density $p_0$.
    Then
    \begin{equation*}
      H(\widehat{p}_{n,s} , p_0) =
      O_p(n^{-2/5}).
    \end{equation*}
  \end{theorem}
  \smallskip

  Theorem~\ref{thm:HellingerConvergenceRate}
  is a fairly straightforward consequence of Theorem~\ref{thm:BracketEntropyBound} by
  applying \cite{MR1739079}, Theorem 7.4, page 99, or
  \cite{MR1385671}, Theorem 3.4.4 in conjunction with Theorem 3.4.1, pages
  322-323.
  \label{note:riskparagraph}

  In the case $s=0$, one can extend our results (an upper bound on the
  rate of convergence) to an upper bound on the risk
  $E_{p_0}(H^2(\widehat{p}_{n,0}, p_0) )$ over the entire class of log-concave densities
  $p_0$; %
  \cite{Kim:2014wa} show how this can be done; they use the fact that the
  log-concave density class is compact in the sense that one can translate
  and rescale to have e.g.\ any fixed mean and covariance matrix one would
  like (since the Hellinger metric is invariant under translation and
  rescaling), and the class of densities with fixed mean and variance is
  uniformly bounded above.  However, to show the risk bound for $s=0$,
  \cite{Kim:2014wa} use many convergence results that are available for
  $0$-concave densities but not yet available for $s$-concave densities with
  $s< 0$.
  In particular, their crucial Lemma 11, page 33, relies on results concerning the asymptotic
behavior of the MLE beyond the log-concave model
${\cal P}_0$ due to D\"umbgen, Samworth, and Schumacher (2011).
We do not yet know if such a result holds for the MLE in any of the classes
${\cal P}_s$ with $s<0$.
  Thus, for the moment, we leave our results as rates of convergence
  rather than risk bounds.

  In addition to Theorem~\ref{thm:HellingerConvergenceRate}, we have further
  consequences since the Hellinger metric dominates the total variation or
  $L_1-$metric and via \cite{MR1739079}, Corollary 7.5, page 100:

  \begin{corollary}
    \label{cor:L1convergenceRate}
    Let Assumption \ref{assump:TrueDensityHypothesis-s} holds and let $ s >
    -1$.  %
    Suppose that $\widehat{p}_{n,s}$ is the MLE of the $s$-concave density
    $p_0$.  Then
    \begin{equation*}
      \int_\RR | \widehat{p}_{n,s} (x) - p_0(x) | \, dx =
      O_p(n^{-2/5}).
    \end{equation*}
  \end{corollary}

  \begin{corollary}
    \label{cor:LogLRconvergenceRate}
    Let Assumptions \ref{assump:TrueDensityHypothesis-s} hold %
    and let $s > -1.$ %
    Suppose that $\widehat{p}_{n,s}$ is the MLE of the $s$-concave density $p_0$.
    Then the log-likelihood ratio (divided by $n$) $\PP_n \log (\widehat{p}_{n,s}
    /p_0) $ satisfies
    \begin{equation}
      \PP_n \log \left ( \frac{\widehat{p}_{n,s}}{p_0} \right ) =
      O_p(n^{-4/5}).
      \label{LogLRBigOhPnMinusFourFifths}
    \end{equation}
  \end{corollary}

  The result (\ref{LogLRBigOhPnMinusFourFifths}) is of interest in connection
  with the study of likelihood ratio statistics for tests (and resulting
  confidence intervals) for the mode $m_0$ of $p_0$ which are being developed
  by the first author.  In fact, the conclusions of
  Theorem~\ref{thm:HellingerConvergenceRate} and
  Corollary~\ref{cor:LogLRconvergenceRate} are also true for the {\sl
    constrained maximum likelihood estimator} $\widehat{p}_n^0$ of $p_0$
  constrained to having (known) mode at $0$.  We will not treat this here, but
  details will be provided along with the development of these tests in
  \cite{doss-phd} and \cite{Doss-Wellner:2013Mode}.

  The rates we have given are for the Hellinger distance (as well as any
  distance smaller than the Hellinger distance) and also for the
  log-likelihood ratio.  The Hellinger metric is very natural for maximum
  likelihood estimation given i.i.d.\ observations, and thus many results are
  stated in terms of Hellinger distance (e.g., \cite{MR1739079} focuses much
  attention on Hellinger distance).  Use of the Hellinger metric is not
  imperative e.g., Theorem 3.4.1 of \cite{MR1385671} %
  is stated for a general metric, but getting rates for other metrics (e.g.,
  $L_r$ for $r > 1$) would require additional work since using Theorem 3.4.1
  of \cite{MR1385671} %
  requires verification of additional conditions which are not immediate.

  Estimators based on shape constraints have been shown to have a wide range
  of adaptivity properties.  For instance, \cite{MR2546798} %
  study the sup-norm on compacta (which we expect to behave differently than
  Hellinger distance) and show that the log-concave MLE is rate-adaptive to
  H\"older smoothness $\beta$ when $\beta \in [1,2]$.  In the case of
  univariate convex regression, \cite{Guntuboyina:2013tl} were able to
  show that the least-squares estimator achieves a
  parametric rate (up to log factors) at piecewise linear functions $\vp_0$.
  They do this by computing entropy bounds for local classes of convex
  functions within a distance $\delta$ of the true function.  We have not yet
  succeeded in extending the bracketing entropy bound of our
  Theorem~\ref{thm:BracketEntropyBound} to analogous local classes, because
  the proof method used for our theorem does not keep tight enough control of
  concave-function classes that do {\em not} drop to $0$ except near a
  pre-specified boundary (where one expects the entropies to be smaller).
  It
  seems that techniques more similar to
  those used by \cite{MR2519658}  %
  or
  \cite{Gunt-Sen:12} may be applicable.
  \section{Main Results: general $h$-transformed classes}
  \label{sec:Appendix0} %

  Here we state and prove the main results of the paper in their most general
  form, via arbitrary concave-function transformations, $h$.
  Similarly to our definition of $\mathcal{P}_s$, we define
  \begin{equation}
    \mathcal{P}_h :=
    \left \{ h \circ \mathcal{C} \right \} \cap \left\{ p : \int p \,
      d\lambda = 1\right\},
    \label{eq:defn:Ph}
  \end{equation}
  the class of $h$-concave-transformed densities, and we study the MLE over
  $\mathcal{P}_h$.  These will be described in more detail in
  Definition~\ref{defn:transformation} and
  Assumption~\ref{assum:transformation}.  In order to study rates of
  convergence, we need to bound bracketing entropies of relevant function
  classes.  Control of the entropies of classes of concave (or convex)
  functions with respect to supremum metrics requires control of Lipschitz
  constants, which we do not have. Thus, we will use $L_r$ metrics with $r
  \ge 1$.
  First, we will define the
  classes of concave and
  concave-transformed functions which we will be studying.

  While we consider $\varphi \in \mathcal{C}$ to be defined on $\RR$, we
  will still sometimes consider a function $\psi$ which is the
  ``restriction of $\varphi$ to $I$'' for an interval $I \subset \RR$.  By this,
  in keeping with the above-mentioned convention, we still mean that $\psi$
  is defined on $\RR$, where if $x \notin I$ then $\psi(x) = -\infty$, and
  otherwise $\psi(x) = \varphi(x)$. We will let $\varphi \vert_I$ denote
  such restricted functions $\psi$.
  When we want to speak
  about the range of any function $f$ (not necessarily concave) we
  will use set notation, e.g.\ for $S \subseteq \RR$,
  $ f(S) := \{ y :  f(x)=y \mbox{ for some } x \in S \}$.
  We will sometimes want to restrict not the domain
  of $\varphi$ but, rather, the range of $\varphi$.  We will thus let
  $\varphi \vert^I$ denote $\varphi \vert_{D_{\varphi,I}}$ for any interval
  $I \subset \RR$, where $D_{\varphi,I} = \{x : \varphi(x) \in I \}.$ Thus,
  for instance, for all intervals $I$ containing
  $\vp (\dom \vp)$
  we have
  $\varphi\vert^I \equiv \varphi$.

  We will be considering classes of nonnegative concave-transformed functions
  of the type $h \circ \mathcal{C}$ for some transformation $h$ where
  $h(-\infty) = 0$ and $h(\infty) = \infty$.
  We will elaborate on these transformations shortly, in
  Definition~\ref{defn:transformation} and
  Assumption~\ref{assum:transformation}.  We will slightly abuse notation by
  allowing the $\dom$ operator to apply to such concave-transformed
  functions, by letting $\dom h\circ \varphi := \{x : h(\varphi(x)) > 0\}$ be
  the support of $h \circ \vp$.

  The function classes in which we will be interested in the end are the
  classes $\Pdms$ defined in \eqref{eq:DefnOfCalPMOne}, or, more  generally $\Pdmh$
  defined in
  \eqref{eq:DefnOfCalPMOneH},
  to which the MLEs (of translated and rescaled data) belong, for some $M< \infty$,
  with high probability as sample size gets
  large.   However, such classes contain functions that are arbitrarily close
  to or equal to $0$ on the support of the true density $p_0$ , and these
  correspond to concave functions that take unboundedly large (negative)
  values on the support of $p_0$. Thus the corresponding concave classes do
  not have finite bracketing entropy for the $L_r$ distance.  To get around
  this difficulty, we will consider classes of truncated concave functions
  and the corresponding concave-transformed classes.

  \begin{definition}
    \label{defn:transformation}
    A {\em concave-function transformation},
    $h$, is a
    continuously differentiable %
    increasing %
    function from
    $[-\infty,\infty]$ to $[0,\infty]$ such that $h(\infty) = \infty $ and
    $h(-\infty) = 0$. We define its limit points $\tilde{y}_0 < \tilde{y}_\infty$ by
    $\tilde{y}_0 = \inf \{y : h(y) > 0 \}$ and
    $\tilde{y}_\infty = \sup \{y : h(y) < \infty\}$, we assume that
    $h(\tilde{y}_0) = 0$ and $h(\tilde{y}_\infty) = \infty$. %
  \end{definition}

  \begin{mynotestoself}
    \begin{notestoself}
      We could allow $h$ to be discontinuous at $\tilde{y}_0$ if $\tilde{y}_0 $
      is not $-\infty$, so $h_0$ is the minimum value of $h$ that is not $0$.
      This allows, for instance, for bounding the height of the density being
      estimated, a technique that has been used as a form of regularization in
      some cases. %
    \end{notestoself}
  \end{mynotestoself}

  \begin{remark}
    These transformations correspond to ``decreasing transformations'' in the
    terminology of \cite{MR2766867}. %
    In that paper, the transformations are applied to convex functions whereas
    here we apply our transformations to concave ones. Since negatives of
    convex functions are concave, and vice versa, each of our transformations
    $h$ defines a decreasing transformation $\tilde{h}$ as defined in
    \cite{MR2766867} %
    via $\tilde{h}(y) = h(-y)$.
  \end{remark}
  We will sometimes make the following assumptions.
  \begin{assumption}[Consistency Assumptions on $h$]
    \label{assum:transformation}
    Assume that the transformation $h$ satisfies:
    \begin{enumerate}[label=T.\arabic{*}]
    \item \label{assum:enum:y0-infinite}
      $h'(y) = o(|y|^{- (\alpha+1)} )$ as %
      $y \searrow -\infty $
      for some
      $\alpha > 1$.
    \item \label{assum:enum:uniflip} If $\tilde{y}_0 >-\infty$, then for all
      $\tilde{y}_0 <c < \tilde{y}_\infty$, there is an $0< M_c < \infty$ such that
      $h'(y) \le M_c$ for all $y \in (\tilde{y}_0, c]$;
    \item \label{assum:enum:yinfty-finite} If $\tilde{y}_\infty < \infty$ then for some
      $0<c<C$, $c (\tilde{y}_\infty - y)^{-\beta} \le h(y) \le      C
      (\tilde{y}_\infty - y)^{-\beta}$
      for some $\beta > 1$ and $y$ in a neighborhood of $\tilde{y}_\infty$;
    \item \label{assum:enum:yinfty-infinite} If $\tilde{y}_\infty = \infty$ then
      $h(y)^{\gamma} h(-Cy) = o(1)$ for some $\gamma, C > 0$, as $y \to \infty$.
    \end{enumerate}
  \end{assumption}

  \begin{oldtext}
    Note that Assumption~\eqref{assum:enum:uniflip} does not preclude $h$ from
    being discontinuous at $\tilde{y}_0$ when $\tilde{y}_0 < \infty$.
    Additionally, notice this assumption holds automatically if
    $\tilde{y}_0 = -\infty$ when Assumption~\eqref{assum:enum:y0-infinite}
    holds.
  \end{oldtext}

  \begin{example}
    \label{ex:LogConcaveTransf}
    The class of log-concave densities, as discussed in
    Section~\ref{sec:BracketEntropyBounds} is obtained by taking $h(y)=e^y
    \equiv h_0 (y)$ for $y \in \RR$.  Then $\tilde{y}_0 = - \infty$ and
    $\tilde{y}_\infty = \infty$.
    Assumption~\eqref{assum:enum:yinfty-infinite} holds with any $\gamma > C
    > 0$, and Assumption~\eqref{assum:enum:y0-infinite} holds for any
    $\alpha  > 1$.
  \end{example}

  \begin{example}
    \label{ex:SConcaveTransf}
    The classes $\mathcal{P}_s$ of $s$-concave densities with $s \in (-1,0)$, as discussed in
    Section~\ref{sec:BracketEntropyBounds}, are obtained by taking $h(y) =
    (-y)_+^{1/s} \equiv h_s (y)$ for $s \in (-1,0)$ and for $y < 0$.  Here
    $\tilde{y}_0 = -\infty$ and $\tilde{y}_\infty = 0$.
    Assumption~\eqref{assum:enum:yinfty-finite} holds for $\beta = -1/s$, and
    Assumption~\eqref{assum:enum:y0-infinite} holds for any $\alpha \in
    (1,-1/s)$.

    Note that the same classes of densities ${\cal P}_s$ result from the
    transforms $\tilde{h}_s (y) = (1+sy)_{+}^{1/s}$ for $y \in (-\infty,
    -1/s) = (\tilde y_0, \tilde y_\infty)$: if $p = h_s (\varphi) \in {\cal
      P}_s$, then also $p = \tilde{h}_s ( \tilde{\varphi}_s ) \in {\cal P}_s$
    where $\tilde{\varphi}_s \equiv - (\varphi + 1)/s $ is also concave.
    With this form of the transformation we clearly have $\tilde{h}_s (y)
    \rightarrow e^{y}$ as $s \nearrow 0$, connecting this example with
    Example 4.1.
  \end{example}

  \begin{example}
    \label{ex:SConcaveTransfPositiveS}
    The classes of $s$-concave functions with $0 < s < \infty $, as discussed
    in Section~\ref{sec:BracketEntropyBounds} are obtained by taking $h(y) =
    (y)_+^{1/s} \equiv h_s (y)$.  Here $\tilde{y}_0 = 0$ and
    $\tilde{y}_\infty = \infty$.  Assumption~\eqref{assum:enum:y0-infinite}
    holds for any $\alpha > 1$, Assumption~\eqref{assum:enum:uniflip}
    fails if $s>1$, and Assumption~\eqref{assum:enum:yinfty-infinite} holds
    for any (small) $C, \gamma >0$.  These (small) classes ${\cal P}_h$ are
    covered by our Corollary~\ref{cor:SConcavePositiveConnectCor}.
  \end{example}

  \begin{example}
    \label{ex:HybridExample}
    To illustrate the possibilities further, consider $h(y) = \tilde{h}_s( y)
    = (1+ sy)^{1/s}$ for $y \in [0,-1/s)$ with $-1< s < 0$, and $h(y) =
    \tilde{h}_r (y) $ for $y \in (-\infty, 0)$ and $r \in (-1,0]$.  Here
    $\tilde{y}_0 = -\infty$ and $\tilde{y}_\infty = -1/s$.
    Assumption~\eqref{assum:enum:yinfty-finite} holds for $\beta = -1/s$, and
    Assumption~\eqref{assum:enum:y0-infinite} holds for any $\alpha \in (1,
    -1/r)$.  Note that $s=0$ is not allowed in this example, since then if $r
    < 0$, Assumption~\eqref{assum:enum:yinfty-infinite} fails.
  \end{example}

  The following lemma shows that concave-transformed classes
  yield nested families ${\cal P}_h$ much as the $s$-concave classes are
  nested, as was noticed in Section~\ref{sec:intro}.

  \begin{lemma}
    \label{lem:convex-ordered-classes}
    Let $h_1$ and $h_2 $ be concave-function transformations.  If $\Psi $ is a
    concave function such that $h_1 = h_2 \circ \Psi$, then
    ${\cal P}_{h_1} \subseteq {\cal P}_{h_2}$.
  \end{lemma}
  \begin{proof}
    Lemma~2.5, page 6, of \cite{MR2766867} %
    gives this result, in the notation of ``decreasing (convex) transformations.''
  \end{proof}

  Now, for an interval $I \subset \RR$, let
  \begin{equation*}
    \mathcal{C}\lp I, [-B,B] \rp
    = \{ \vp \in
    \mathcal{C} :  -B \le \vp(x) \le B
    \mbox{ if } x \in \dom \vp = I \}.
  \end{equation*}
  Despite making no restrictions on the Lipschitz behavior of the function
  class, we can still bound the entropy, as long as our metric is $L_r$ with
  $1 \le r < \infty$ rather than $L_\infty$.
  \begin{proposition}[Extension of Theorem~3.1 of \cite{Gunt-Sen:12}]
    \label{prop:GSextension}
    Let $b_1 < b_2$.
    Then there exists
    %
    $C < \infty$ such that
    \begin{equation}
      \log N_{[\, ]} (\epsilon, \mathcal{C}([b_1,b_2], [-B,B]), L_r)
      \le C \lp\frac{B (b_2-b_1)^{1/r}}{\epsilon} \rp^{1/2}
      \label{eq:extensionGSbound}
    \end{equation}
    for all $\epsilon > 0 $.  
  \end{proposition}
  Our first main result has a statement analogous
  to that of the previous proposition, but it is not about concave or convex
  classes of functions but rather about concave-transformed classes, defined
  as follows.  Let $h$ be a concave-function transformation. Let
  $\mathcal{I}[b_1,b_2]$ be all intervals $I$ contained in $[b_1, b_2]$, and
  let
  \begin{equation*}
    \mathcal{F} \lp  \mathcal{I}[b_1,b_2]  , [0,B] \rp
    = \lb f  : f = h \circ \vp, \vp \in \mathcal{C}, \dom \vp \subset
    [b_1,b_2], 0 \le f \le B \rb
  \end{equation*}

  \begin{theorem}
    \label{thm:entropy_compactx}
    Let $r \ge 1$. Assume $h$ is a concave-function transformation.  If
    $\tilde y_0 = -\infty$ then assume $h'(y) = o( |y|^{-(\alpha+1)})$ for
    some $\alpha > 0$ as $y \to -\infty$.  Otherwise assume
    Assumption~\eqref{assum:enum:uniflip} holds.
    Then      for all $\epsilon >0 $ 
    \begin{equation*}
      \frac{\log N_{[\, ]}(\epsilon ,
        \mathcal{F} \lp \mathcal{I}[b_1,b_2], [0,B] \rp,
        L_r)}{ (B (b_2-b_1)^{1/r})^{1/2}}
      \lesssim
      \epsilon^{-1/2}
    \end{equation*}
    where $\lesssim$ means $\le$ up to a constant.  The constant implied by
    $\lesssim$ depends only on $r$ and $h$.
  \end{theorem}
  Thus, a bounded class of transformed-functions for any reasonable
  transformation behaves like a compact class of concave functions.

  We extend the
  definition \eqref{eq:DefnOfCalPMOne}
  to an arbitrary concave-function transformation $h$ as follows:
  \begin{equation}
    \Pdmh
    \equiv
    \lb p \in \mathcal{P}_{h} \colon \sup_{x \in \RR} p(x) \le M, 1/M \le p(x)  \mbox{
      for all } \vert x \vert \le 1  \rb.
    \label{eq:DefnOfCalPMOneH}
  \end{equation}
  As with the analogous classes of log-concave and $s$-concave
 densities,
  the class $\mathcal{P}_{M,h}$ is important because it has an upper
  envelope, which is given in the following proposition.

  \begin{proposition} %
    \label{prop:calPMOnehEnvelope}
    Let $h$ be a concave-function transformation such that
    Assumption~\eqref{assum:enum:y0-infinite} holds with exponent $\alpha_h >
    1$.
    Then for any %
    $p^{1/2} \in \Pmh^{1/2}$
    with $0 < M <\infty$,
    \begin{eqnarray}
      p^{1/2}(x)
      & \le & \left \{ \begin{array}{l l}
          D^{1/2} \left ( 1  +  \frac{L}{2M} |x|  \right )^{-\alpha_h/2}, & |x| \ge 2M+1, \\
          M^{1/2} , & |x| < 2M+1 %
        \end{array} \right \}\nonumber \\
      & \equiv & p_{u,h}^{1/2} (x),
      \label{eq:DefnOfCalPMOnehEnvelope}
    \end{eqnarray}
    where $0 < D, L < \infty$ are constants depending only on $h$ and $M$.
  \end{proposition}

 We would like to bound the bracketing entropy of the classes $\Pmh$.  This
requires allowing possibly
 unbounded support.  To do this, we will apply
the envelope from the
 previous proposition and then apply
Theorem~\ref{thm:entropy_compactx}.
 Because the size or cardinality of the
brackets depends on the height of
 the function class, the upper bound on
the heights given by the envelope
 allows us to take brackets of
correspondingly decreasing size and
 cardinality out towards infinity.
Combining all the brackets from the
  partition of $\RR$ yields the result.  Before we state the theorem, we need
  the following assumption, which is the more general version of Assumption
  \ref{assump:TrueDensityHypothesis-s}.
  \begin{assumption}
    \label{assump:TrueDensityHypothesis-h}
    We assume that $X_{i}$, $i=1,\ldots, n$ are i.i.d.\ random variables
    having density
    $p_0 = h \circ \varphi_0 \in {\cal P}_{h}$ where $h$ is a
    concave-function transformation.
  \end{assumption}

  \begin{theorem}
    \label{thm:entropy_noncompactx_RRd}
    Let $r \ge 1$, $M > 0$, and $\epsilon > 0$.  Let $h$ be a
    concave-function transformation such that for $g \equiv h^{1/2}$,
    Assumption~\ref{assum:transformation},\eqref{assum:enum:y0-infinite}-\eqref{assum:enum:yinfty-infinite}
    hold, with $\alpha \equiv \alpha_g > 1/r \vee 1/2$.
    Then
    \begin{equation}
      \label{eq:bracketF}
      \log N_{[\, ]}(\epsilon, \mathcal{P}_{M, h}^{1/2}, L_r)
      \le
      K_{r, M, h}
      \epsilon^{-1/2}.
    \end{equation}
    where $K_{r,M,h}$ is a constant depending only on $r$, $M$, and $h$.
  \end{theorem}

  For the proof of this theorem (given with the other proofs, in
  Section~\ref{sec:mainresults-proofs}), we will pick a sequence
  $y_\gamma,$ for $\gamma = 1, \ldots k_{\epsilon}$ to discretize the range
  of values that a concave function $\vp$ may take, where $k_\epsilon$
  defines the index of truncation which necessarily depends on $\epsilon$ in
  order to control the fineness of the approximation.  This allows us to
  approximate a concave function $\vp$ more coarsely as $y_\gamma$ decreases,
  corresponding to approximating the corresponding concave-{\em transformed}
  function $h \circ \vp$ at the {\em same} level of fineness at all
  $y_\gamma$ levels.


  %
  %

  \begin{remark}
    We require that $h^{1/2}$, rather than $h$ itself, is a concave-function
    transformation here because to control Hellinger distance for the class
    ${\cal P}_{M,h}$, we need to control $L_2$ distance for the class ${\cal
      P}_{M,h}^{1/2}$.  Note that when $h$ is $h_s$ for any $s \in \RR$,
    $h^{1/2}$ is also a concave-function transformation.
  \end{remark}

  We can now state our main rate result theorem, which is the general form of
  Theorem~\ref{thm:HellingerConvergenceRate}.  It is proved by using
  Theorem~\ref{thm:entropy_noncompactx_RRd}, specifying to the case $r=2$.
  There is seemingly a factor of two different in the assumptions for the
  $s$-concave rate theorem (requiring $-1/s > 1$) and the assumption in the
  $h$-concave rate theorem, requiring $\alpha > 1/2$ (where, intuitively, we
  might think $\alpha$ corresponds to $-1/s$).  The reason for this
  discrepancy is that $\alpha $ in the $h$-concave theorem is $\alpha_g$
  corresponding to $g \equiv h^{1/2}$, rather than corresponding to $h$
  itself; thus $\alpha_g$ corresponds not to $(-1/s)$ but to $(-1/s)/2$.

  \begin{theorem}
    \label{thm:HellingerRateTheoremFinal}
    Let Assumption~\ref{assump:TrueDensityHypothesis-h} hold and let
    $\widehat{p}_n$ be the %
    $h$-transformed %
    MLE of $p_0$.  Suppose that
    Assumption~\ref{assum:transformation},\eqref{assum:enum:y0-infinite}-\eqref{assum:enum:yinfty-infinite}
    holds for $g \equiv h^{1/2}$.  Assume that $\alpha \equiv \alpha_g > 1/2$.
    Then
    \begin{equation}
      H(\widehat{p}_n , p_0) =
      O_p(n^{-2/5}).
      \label{eq:hellinger_stdrate_anysupport_UC_MC}
    \end{equation}
  \end{theorem}

  The following corollaries connect the general
  Theorem~\ref{thm:HellingerRateTheoremFinal} with
  Theorem~\ref{thm:HellingerConvergenceRate} via
  Examples~\ref{ex:LogConcaveTransf}, ~\ref{ex:SConcaveTransf}, and
  \ref{ex:SConcaveTransfPositiveS}.

  \begin{corollary}
    \label{cor:LogConcaveConnectCor}
    Suppose that $p_0$ in Assumption~\ref{assump:TrueDensityHypothesis-h} is
    log-concave; that is, $p_0 = h_0 \circ \varphi_0$ with $h_0(y) = e^y$ as
    in Example~\ref{ex:LogConcaveTransf} and $\varphi_0$ concave.  Let
    $\widehat p_n$ be the MLE of $p_0$.  Then $H(
    \widehat{p}_n , p_0 ) = O_p (n^{-2/5})$.
  \end{corollary}

  \begin{corollary}
    \label{cor:SConcaveConnectCor}
    Suppose that $p_0$ in Assumption~\ref{assump:TrueDensityHypothesis-h} is
    $s$-concave with $-1<s <0$; that is, $p_0 = h_s \circ \varphi_0$ with
    $h_s(y) = (-y)^{1/s}$ for $y<0$ as in Example~\ref{ex:SConcaveTransf}
    with $-1<s <0$ and $\varphi_0$ concave.  Let $\widehat p_n$ be the MLE
    of $p_0$.  Then $H( \widehat{p}_n , p_0 ) = O_p (n^{-2/5})$.
  \end{corollary}

  \begin{corollary}
    \label{cor:SConcavePositiveConnectCor}
    Suppose that $p_0$ in Assumption~\ref{assump:TrueDensityHypothesis-h} is
    $h$-concave where $h$ is a concave tranformation satisfying
    Assumption~\ref{assum:transformation}.  Suppose that $h$ satisfies $h = h_2
    \circ \Psi$ where $\Psi$ is a concave function and $h_2$ is a concave-function
    transformation such that $g \equiv h_2^{1/2}$ also satisfies
    Assumption~\ref{assum:transformation}, and such that
    $\alpha \equiv \alpha_g > 1/2$.
    Let $\widehat p_n$ be the $h$-concave MLE of $p_0$.  Then
    \begin{equation}
      H(\widehat{p}_n , p_0) = O_p(n^{-2/5}).
      \label{eq:hellinger_stdrate_anysupport_UC_MC-positive-s}
    \end{equation}
    In particular the conclusion holds for $h=h_s$ given by $h_s (y) =
    y_+^{1/s}$ with $s>0$.
  \end{corollary}

  Corollaries \ref{cor:LogConcaveConnectCor} and \ref{cor:SConcaveConnectCor}
  follow immediately from Theorem~\ref{thm:HellingerRateTheoremFinal} (see
  Examples \ref{ex:LogConcaveTransf} and ~\ref{ex:SConcaveTransf}).  However
  Corollary~\ref{cor:SConcavePositiveConnectCor} requires an additional
  argument (given in the proofs section). Together, these three corollaries
  yield Theorem~\ref{thm:HellingerConvergenceRate} in the main document.

  Theorem~\ref{thm:HellingerRateTheoremFinal}  has further corollaries, for example
  via Example~\ref{ex:HybridExample}.

\section{Summary, further problems, and prospects}
\label{sec:SummaryProblemsProspects}

In this paper we have shown that the MLE's of $s-$concave densities on $\RR$ have Hellinger convergence rates
of $n^{-2/5}$ for all $s>-1$ and that the MLE does not exist for $s<-1$.
Our bracketing entropy bounds explicitly quantify the growth
of these classes as $s \searrow -1$ and are of independent interest in the study of
convergence rates for other possible estimation methods.   In the rest of this section we briefly discuss
some further problems.

\subsection{Behavior of the constants in our bounds}

It can be seen from the proof of 
Theorem~\ref{thm:entropy_noncompactx_RRd} that the constants in our entropy
bounds diverge to $+\infty$ as $\alpha = \alpha_g \searrow 1/r$.  When
translated to Theorem~\ref{thm:BracketEntropyBound} and $r=2$ this occurs as
$(-1/(2s)) \searrow 1/2$.  It would be of interest to establish lower bounds
for these entropy numbers with the same property.  On the other hand, when
$r=2$ and $s= -1/2$, the constant $K_{r,\alpha}$ in the proof of
Theorem~\ref{thm:entropy_noncompactx_RRd}
becomes very reasonable: $K_{2,1} = M^{1/5} (4 M + 2)^{1/5} + 16
(2D^{1/2}M/L)^{2/5}$ where $M, D$, and $L$ are the constants in the envelope
function $p_{u,h}$ of Proposition~\ref{prop:calPMOnehEnvelope}.  Note that
the constant $ \tilde K_{r,\alpha} $ from Theorem~\ref{thm:entropy_compactx}
arises as a factor in the constant for
Theorem~\ref{thm:entropy_noncompactx_RRd}, but from the proof of
Theorem~\ref{thm:entropy_compactx} it can be seen that unless $\alpha
\searrow 0$, $\tilde K_{r,\alpha}$ stays bounded. 



\subsection{Alternatives to Maximum likelihood}
As noted by
\cite{MR2722462}, page 2999,  %
there are great algorithmic advantages in adapting the method of estimation to the particular
class of shape constraints involved, thereby achieving a convex optimization problem with a tractable
computational strategy.  In particular,
\cite{MR2722462} %
showed how R\'{e}nyi divergence methods are well-adapted
to the $s-$concave classes in this regard.
As has become clear through the work of
\cite{Han-Wellner:2015},  %
there are further advantages in terms of robustness and stability properties of the alternative
estimation procedures obtained in this way.
\subsection{Rates of convergence for nonparametric estimators, $d\ge 2$}
Here we have provided global rate results for the MLEs over ${\cal P}_{1,s}$
with respect to the Hellinger metric.  Global rate results are still lacking
for the classes ${\cal P}_{d,s}$ on $\RR^d$ with $d \ge 2$.
\cite{Kim:2014wa} provides interesting and important minimax lower bounds for
squared Hellinger risks for the classes ${\cal P}_{d,0}$ with $d \ge 1$, and
their lower bounds apply to the classes ${\cal P}_{d,s}$ as well in view of
the nesting properties in (\ref{sConcaveNestingProperty}) and Lemma 4.1.
Establishment of comparable upper bounds for $d\ge 2$ remains an active area
of research.

%

%
%
\subsection{Rates of convergence for the R\'{e}nyi divergence estimators}
Although global rates of convergence of the R\'{e}nyi divergence estimators of \cite{MR2722462}
have not yet been established even for $d=1$, we believe that the bracketing entropy bounds obtained
here will be useful in establishing such rates.  The results of
\cite{Han-Wellner:2015} provide some useful starting points in this regard.
\subsection{Global rates of convergence for density estimation in $L_1$}
Rates of convergence with respect to the $L_1$ metric
for MLE's for the classes ${\cal P}_{d,0}$ and ${\cal P}_{d,s}$ with $d\ge 2$ and $s<0$
are not yet available.  At present, further tools seem to be needed.
\subsection{Rate efficient estimators when $d\ge 3$} It has become
increasingly clear that nonparametric estimators based on minimum contrast
methods (either MLE or minimum R\'{e}nyi divergence) for the classes
${\cal P}_{d,s}$ with $d\ge 3$ will be rate inefficient.
This modified form of the conjecture of \cite{MR2766867}, section 2.6, page 3762, accounts
for the fact pointed out by \cite{Kim:2014wa} 
that the classes ${\cal P}_{d,s}$ with $-1/d< s \le 0$
contain all uniform densities on compact convex subsets of $\RR^d$, and these
densities have Hellinger entropies of order $\epsilon^{-(d-1)}$.
Hence alternative
procedures based on sieves or penalization will be required to achieve
optimal rates of convergence.  Although these problems have not yet been
pursued in the context of log-concave and $s-$concave densities, there is
related work by \cite{MR3014311},  %
in a closely related problem involving estimation of the support functions of
convex sets.

  \section{Main Results: Proofs}
  \label{sec:mainresults-proofs}

  This section contains the proofs of the main results. %
  \begin{proof}[Proof of Proposition~\ref{prop:MLEnonexistence}]
    Let $s<-1$ and set $r \equiv -1/s < 1$.  Consider the family of convex functions $\{ \varphi_a \}$
    given by
    $$
    \varphi_a(x) = a^{-1/r} (b_r -ax) 1_{[0,b_r/a]} (x)
    $$
    where $b_r \equiv (1-r)^{1/(1-r)}$ and $a>0$.
    Then $\varphi_a$ is convex and
    $$
    p_a (x) \equiv \varphi_a (x)^{1/s} = \varphi_a (x)^{-r} = \frac{a}{(b_r-ax)^r} 1_{[0,b_r/a)} (x)
    $$
    is a density.
    The log-likelihood is given by
    \begin{eqnarray*}
      \ell_n (a) & = & \log L_n (a) = \log \prod_{i=1}^n p_a (X_i)
      =  \sum_{i=1}^n \{ \log a - r \log (b_r - a X_i ) \}
    \end{eqnarray*}
    on the set $X_i < b_r/a$ for all $i \le n$ and hence for $a <
    b_r/X_{(n)}$ where $X_{(n)} \equiv \max_{1 \le i \le n} X_i$.  Note that
    $\ell_n (a) \nearrow \infty$ as $a \nearrow b_r/X_{(n)}$.  Hence the MLE
    does note exist for $\{ p_a : a > 0 \}$, and a fortiori the MLE does not
    exist for $\{ p : \ p \in {\cal P}_{1,s} \}$ with $s< -1$.
  \end{proof}

  \begin{proof}[Proof of Proposition~\ref{prop:GSextension}]
    The proof consists mostly of noticing that Theorem 3.1 in
    \cite{Gunt-Sen:12} essentially yields the result stated here; the
    difference in the statements is that we use $L_r$ bracketing entropy
    whereas they use $L_r$ metric entropy.  For the details of the proof, see
    Section~\ref{sec:Appendix1}.
  \end{proof}

  To prove Theorem~\ref{thm:entropy_compactx}, we discretize the domains and
  the range of the concave-transformed functions.  We define a sequence of
  values $y_\gamma$ that discretize the range of the concave functions.  As
  $|y_{\gamma}|$ get large, $h(y_{\gamma})$ get small, so we can define
  brackets of increasing size.  The increasing size of the brackets will be
  governed by the values of $\epsilon_\gamma^B$ in the proof.  We also have
  to discretize the domain of the functions to allow for regions where the
  concave-transformed functions can become $0$ (which corresponds to concave
  functions becoming infinite, and which thus cannot be bracketed at the
  concave level).  The sizes of the discretization of the domain
  corresponding to each level $y_\gamma$ is governed by the values of
  $\epsilon_\gamma^S$ in the proof.

  \begin{proof}[Proof of Theorem~\ref{thm:entropy_compactx}]

    First note that the $L_r$ bracketing numbers scale in the following
    fashion.
    For a function $f$
    supported on a subset of $[b_1,b_2]$ and with $|f|$ bounded by $B$, we
    can define a scaled and translated version of $f$,
    \begin{equation*}
      \tilde f(x) := \frac{f(b_1  + (b_2-b_1) x) }{B},
    \end{equation*}
    which is supported on a subset of $[0,1]$ and bounded by $1$.  Then
    \begin{equation*}
      B^r \int_{[0,1]} \left\vert \tilde{f}(x) - \tilde{g}(x) \right\vert^r
      dx
      = \inv{(b_2-b_1)} \int_{[b_1,b_2]} \left\vert f(x) - g(x)
      \right\vert^r dx.
    \end{equation*}
    Thus a class of $\epsilon$-sized $L_r$ brackets when $b_1=0$, $b_2=1$, and
    $B=1$ scales to be a class of $\epsilon(b_2-b_1)^{1/r} B $ brackets for
    general $b_1, b_2$, and $B$.  Thus, for the remainder of the proof we take
    $b_1=0$, $b_2=1$, and $B=1$.
    By replacing $h$ by a translation of $h$ (since concave functions plus a
    constant are still concave), and using the fact that the range of $h$ is
    $(0,\infty)$, we assume that $h^{-1}(1) < 0$.

    We will shortly define a sequence of epsilons, $\epsilon_\gamma^B$
    and
    $\epsilon_\gamma^S$, depending on $\epsilon$.
    We will need     $\epsilon_\gamma^S \le 1$ for all $\gamma$.
    Thus
    we will  later specify a constant $\epsilon^*$
    such that $\epsilon \le \epsilon^*$ guarantees $\epsilon_\gamma^S \le 1$.

    We will consider the cases $\tilde{y}_0 = -\infty$ and $\tilde{y}_0 >
    -\infty$ separately; the former case is more difficult, so let us begin
    by assuming that $\tilde{y}_0 = -\infty$.  Let $y_\gamma = -2^{\gamma}$
    for $\gamma =1, \ldots, k_\epsilon \equiv \left \lfloor \log_2 h^{-1}
      (\epsilon) \right \rfloor$.
    The $y_\gamma$'s discretize the range of possible values a concave
    function takes.  We let $\epsilon_\gamma^B = \epsilon
    (-y_{\gamma-1})^{(\alpha+1)\zeta}$ and $\epsilon_\gamma^S = \epsilon^r
    (-y_{\gamma-1})^{r\alpha \zeta}$, where we choose $\zeta$ to satisfy
    $1 > \zeta > 1/(\alpha+1)$.

    We start by discretizing the support $[0,1]$.  At each level $\gamma =
    1, \ldots, k_\epsilon$, we use $\epsilon_\gamma^S$ to discretize the
    support into intervals on which a concave function can cross below
    $y_\gamma$.

    We place $\ceil*{2/ \epsilon_{\gamma}^S}$ points $a_l$ in in $[0,1]$,
    $l=1, \ldots, \ceil*{2 / \epsilon_{\gamma}^S}$, such that $0 <
    a_{l+1}-a_l<\epsilon_\gamma^S/2 $,
    $l=0, \ldots, \ceil*{2 / \epsilon_{\gamma}^S}$
    taking $a_{l_0}=0$ and
    $a_{l_{\ceil*{2/\epsilon_\gamma^S}+1}} = 1$.  There are $ N_\gamma^S
    \equiv { \ceil*{2/ \epsilon_\gamma^S} \choose 2 }$ pairs of the points,
    and for each pair $(l_1,l_2)$ we define a pair of intervals,
    $\IL{i}{\gamma}$ and $\IU{i}{\gamma}$ by
    \begin{equation*}
      \IL{i}{\gamma} = [a_{l_1}, a_{l_2}] \mbox{ and }
      \IU{i}{\gamma} = [a_{l_1-1}, a_{l_2 +1}],
    \end{equation*}
    for $i = 1 , \ldots, N_\gamma^S$.
    We see that
    $\log N_\gamma^S \le 4 \log( 1/ \epsilon_\gamma^S) $, that
    $\lambda( \IU{i}{\gamma} \setminus \IL{i}{\gamma}) \le \epsilon_\gamma^S$
    and that for each $\gamma$, for all intervals $I \subset [0,1]$ (i.e.,
    for all possible domains $I$ of a concave function $\vp \in \cC[ [0,1],
    [-1,1]]$), there exists $1 \le i \le N_\gamma^S$ such that
    $\AL{i}{\gamma} \subseteq I \subseteq \AU{i}{\gamma}$.

    Now,
    we can apply
    Proposition~\ref{prop:GSextension}
    so for each $\gamma = 1, \ldots,
    k_\epsilon$
    we can pick brackets $[\ll{\alpha}{i}{\gamma}(x),
    \uu{\alpha}{i}{\gamma}(x)]$ for $\mathcal{C}(\AL{i}{\gamma}, [y_\gamma,
    y_0])$ with $\alpha = 1,\ldots, \NB{\gamma} = \left\lfloor \exp(C (
      |y_\gamma| / \epsilon_\gamma^B )^{1/2})\right\rfloor$ (since $y_0 \le
    |y_\gamma|$) and $L_r( \ll{\alpha}{i}{\gamma}, \uu{\alpha}{i}{\gamma} )
    \le \epsilon_\gamma^B$.
    Note that by Lemma~\ref{lem:h-inverse},
     $k_\epsilon \le \log_2 M \epsilon^{-1/\alpha}$ for some $M
     \ge 1$, so
    we see that
    \begin{equation*}
      \epsilon_\gamma^S
      \le \epsilon^{ (1-\zeta)r} \lp
      \frac{M}{2}\rp^{r \alpha \zeta},
    \end{equation*}
    and thus taking $\epsilon^* \equiv (2/M)^{\alpha \zeta / (1-\zeta)}$ 
    the above display is bounded above by $1$ for all $\epsilon \le
    \epsilon^*$, as needed.

    Now we can define the brackets for
    $\mathcal{F}( \mathcal{I}[0,1],
    [0,1])$. %
    For multi-indices $\ibf = (i_1, \ldots, i_{k_\epsilon})$ and
    $\abf = \lp \alpha_1, \ldots, \alpha_{k_\epsilon} \rp$, we define
    brackets $[\fuia, \flia]$ by
    \begin{equation*}
      \begin{split}
        \fuia(x) &  = \sum_{\gamma =1 }^{k_\epsilon} \bigg( h\lp
        \uu{\alpha_\gamma}{i_\gamma}{\gamma}(x) \rp
        \mathbbm{1}_{\lb x \in \AL{i_\gamma}{\gamma} \setminus \cup_{j=1}^{\gamma-1}
          \AU{i_j}{j} \rb } \\
        & \quad + h(y_{\gamma-1}) \mathbbm{1}_{\lb x \in \AU{i_\gamma}{\gamma} \setminus \lp
          \cup_{j=1}^\gamma \AL{i_j}{j} \cup_{j=1}^{\gamma-1} \AU{i_j}{j}
          \rp \rb } \bigg)
        + \epsilon \mathbbm{1}_{ \lb x \in [0,1] \setminus
          \cup_{j=1}^{\gamma} \AU{i_j}{j} \rb }, \\
        \flia(x) & = \sum_{\gamma=1}^{k_\epsilon} h\lp
        \ll{\alpha_\gamma}{i_\gamma}{\gamma}(x) \rp
        \mathbbm{1}_{\lb x \in \AL{i_\gamma}{\gamma} \setminus
          \cup_{j=1}^{\gamma-1} \AU{i_j}{j} \rb }.
      \end{split}
    \end{equation*}
    Figure~\ref{fig:bracket-join} gives a plot of $ [\flia,        \fuia]$.
    For $x \in \AL{i_\gamma}{\gamma} \setminus \cup^{\gamma-1}_{j=1}
    \AU{i_j}{j}$, we can assume that $y_\gamma \le
    \uu{i_\gamma}{\alpha_\gamma}{\gamma}(x) \le y_{\gamma-1}$ by replacing
    $\uu{i_\gamma}{\alpha_\gamma}{\gamma}(x)$ by
    $(\uu{i_\gamma}{\alpha_\gamma}{\gamma}(x) \wedge y_{\gamma-1}) \vee
    y_{\gamma}$. We do the same for
    $\ll{i_\gamma}{\alpha_\gamma}{\gamma}(x)$.

    We will check that these do indeed define a set of bracketing functions
    for $\mathcal{F}(\mathcal{I}[0,1], [0,1])$ by considering separately
    the different domains on which $\fuia$ and $\flia$ are defined.  We take
    any $h( \vp) \in \mathcal{F}(\mathcal{I}[0,1], [0,1])$, and then for
    $\gamma = 1, \ldots, k_\epsilon$, we can find $\AL{i_\gamma}{\gamma}
    \subseteq \dom (\vp \vert^{[y_\gamma, \infty)}) \subseteq
    \AU{i_\gamma}{\gamma}$ for some $i_\gamma \le \NS{\gamma}$.  So, in
    particular,
    \begin{equation}
      \label{eq:bracket-dom-prop}
      \vp(x) <        y_\gamma \mbox{ for } x \notin \AU{i_\gamma}{\gamma},
      \mbox{ and }  y_\gamma \le \vp(x) \mbox{ for } x \in \AL{i_\gamma}{\gamma}.
    \end{equation}
    Thus, there is an $\alpha_\gamma$ such that
    $\ll{\alpha_\gamma}{i_\gamma}{\gamma}$ and
    $\uu{\alpha_\gamma}{i_\gamma}{\gamma}$ have the bracketing property for
    $\vp$ on $\AL{i_\gamma}{\gamma}$, by which we mean that for $x \in
    \AL{i_\gamma}{\gamma}$, $\ll{\alpha_\gamma}{i_\gamma}{\gamma}(x) \le
    \vp(x) \le \uu{\alpha_\gamma}{i_\gamma}{\gamma}(x)$.  Thus on the sets
    $\AL{i_\gamma}{\gamma} \setminus \cup^{\gamma-1}_{j=1} \AU{i_j}{j}$,
    the functions $\fuia$ and $\flia$ have the bracketing property for
    $h(\vp)$.  Now, $\flia$ is $0$ everywhere else and so is everywhere
    below $h(\vp)$.  $\fuia$ is everywhere above $h(\vp)$ because for $x
    \in \lp\cup_{j=1}^{\gamma-1} \AU{i_j}{j} \rp^c$, we know $h(\vp(x)) \le
    h(y_{\gamma-1})$ by \eqref{eq:bracket-dom-prop}. It just remains to
    check that $\fuia(x) \ge h(\vp(x)) $ for $x \in [0,1] \setminus
    \cup_{j=1}^{\gamma} \AU{i_j}{j} $, and this follows by the definition
    of $k_\epsilon$ which ensures that $h(y_{k_\epsilon}) \le
    \epsilon$ and from \eqref{eq:bracket-dom-prop}. Thus
    $[\flia, \fuia]$ are indeed brackets for $\mathcal{F}(
    \mathcal{I}[0,1], [0,1])$.

    Next we compute the size of these brackets.  We have that
    $L_r^r(\fuia,\flia)$ is
    \begin{align*}
      \int \lp \fuia - \flia \rp^r d\lambda
      & \le \sum_{\gamma=1}^{k_\epsilon} \int_{ \AL{i_\gamma}{\gamma}
        \setminus \AU{i_{\gamma-1}}{\gamma-1} }
      \lp
      h(\uu{\alpha_\gamma}{i_\gamma}{\gamma})
      - h(\ll{\alpha_\gamma}{i_\gamma}{\gamma}) \rp^r d\lambda  \\
      & \quad + \int_{\AU{i_\gamma}{\gamma} \setminus
        \AL{i_{\gamma}}{\gamma}}
      h(y_{\gamma-1})^r d\lambda
      + \epsilon^r \\
      & \le \sum_{\gamma=1}^{k_\epsilon} \sup_{y \in [y_\gamma,
        y_{\gamma-1}]} h'(y)^r \int_{\AL{i_\gamma}{\gamma} \setminus
        \AU{i_{\gamma-1}}{\gamma-1}}
      \lp \uu{\alpha_\gamma}{i_\gamma}{\gamma} -
      \ll{\alpha_\gamma}{i_\gamma}{\gamma} \rp^r d\lambda \\
      & \quad + \sum^{k_\epsilon}_{\gamma=1} h(y_{\gamma-1})^r
      \epsilon^S_\gamma
      + \epsilon^r,
    \end{align*}
    since we specified the brackets to take values in $[y_\gamma,
    y_{\gamma-1}]$ on $\AL{i_\gamma}{\gamma} \setminus
    \AU{i_{\gamma-1}}{\gamma-1}$.  By
    our assumption that $h'(y) = o( |y|^{-(\alpha+1)})$ (so, additionally,
    $h(y) = o( |y|^{-\alpha})$) as $y \to -\infty$, and the definition of
    $\epsilon^B_\gamma$, the above display is bounded above by
    \begin{align*}
      \epsilon^r
      + \sum^{k_\epsilon}_{\gamma=1}  (-y_{\gamma-1})^{-(\alpha+1)r} \epsilon^r
      (-y_{\gamma-1})^{(\alpha+1)\zeta r} + \epsilon^r (-y_{\gamma-1})^{-\alpha r
        (1-\zeta)}
      \le \tilde{C}_1 \epsilon^r
    \end{align*}
    since $\alpha r (1- \zeta) $ and $(\alpha+1) r (1-\zeta)$ are both
    positive, where $\tilde{C}_1 = (1 + 2 / (1-2^{-\alpha r (1-\zeta)}))$.

    Finally, we can see that the log-cardinality of our set of bracketing
    functions, $\log \prod_{\gamma=1}^{k_\epsilon} \NB{\gamma}
    \NS{\gamma}$, is
    \begin{equation}
      \label{eq:log-cardinality}
      \sum_{\gamma=1}^{k_\epsilon}
      C \lp \frac{|y_\gamma|}{\epsilon^B_\gamma} \rp^{1/2}
      +
      4      \log \lp  \inv{ \epsilon^S_\gamma} \rp,
    \end{equation}
    with $C$ from Proposition~\ref{prop:GSextension}.
    The above display is bounded above by
    \begin{equation*}
      \begin{split}
        \MoveEqLeft
        C \sum_{\gamma=1}^{k_\epsilon}  \frac{2^{\gamma/2}}{\epsilon^{1/2}}
        2^{-(\gamma-1)(\alpha+1) \zeta/2}
        + 4  \log %
        \lp \epsilon^{-r} (-y_{\gamma-1})^{-r \alpha \zeta} \rp   \\
        &          \le
        \lp C \vee 4 \rp
        \lp  
        \sum_{\gamma=0}^\infty
        \frac{2^{-((\alpha+1)\zeta -1)\gamma/2 + 1/2}}{\epsilon^{1/2}}
        +  \sum_{\gamma=0}^\infty  %
        \frac{(-y_{\gamma})^{-\alpha \zeta / 2}}{\epsilon^{1/2}} \rp.
      \end{split}
    \end{equation*}
    Since $ (\alpha+1) \zeta - 1  > 0$,
    the above display is finite and can be bounded by $\tilde C_2
    \epsilon^{-1/2}$ where $\tilde C_2 = (C \vee 4) \lp \frac{2^{3/2}}{1 -
      2^{- (\frac{\alpha+1}{4}-\inv{2})}} \vee \frac{2}{1-2^{-\alpha/4}}
    \rp$.
    We have now shown,  for $\tilde{y}_0 = -\infty$ and $\epsilon
    \le \epsilon^*$ that 
    \begin{equation*}
      \log N_{[\, ]} \lp \epsilon \tilde C_1^{1/r}, \mathcal{F}\lp
      \mathcal{I}[0,1], [0,1]\rp , L_r\rp 
      \le \tilde C_2 \epsilon^{-1/2}
    \end{equation*}
    or for $\varepsilon \le \tilde C_1^{1/r} \epsilon^*$, 
    \begin{equation*}
      \log N_{[\, ]} \lp \varepsilon , \mathcal{F}\lp
      \mathcal{I}[0,1], [0,1]\rp , L_r\rp 
      \le  \tilde K_{r,h}  \varepsilon^{-1/2},
    \end{equation*}
    with $\tilde K_{r,h} \equiv \tilde C_1^{1/(2r)}\tilde C_2$.  We mention
    how to extend to all $\epsilon > 0$ at the end.

    \smallskip

    Now let us consider the simpler case, $\tilde{y}_0>-\infty$.  Here we
    take $k_\epsilon = 1$, $y_0 = h^{-1}(1) < 0$, and $y_1 = h^{-1}(0) =
    \tilde y_0$.  Then we define $\epsilon^B = \epsilon$, take $\epsilon^*
    \le 1$, and $\epsilon^S=
    \epsilon^r \le \epsilon^*$ and we define $\AU{i}{\gamma}$, $\AL{i}{\gamma}$,
    $N_\gamma^S$, $[\ll{\alpha}{i}{\gamma}, \uu{\alpha}{i}{\gamma}]$, and
    $N_\gamma^B$ as before, except we will subsequently drop the $\gamma$
    subscript since it only takes one value.  We can define brackets
    $[f^L_{i,\alpha}, f^U_{i,\alpha}]$ by
    \begin{equation*}
      \begin{split}
        f^U_{i,\alpha}(x) & = h \lp \uu{\alpha}{i}{}(x) \rp
        \mathbbm{1}_{A_i^L}(x)
        + h(y_0) \mathbbm{1}_{A_i^U \setminus A_i^L}(x) \\
        f^L_{i,\alpha}(x) & = h\lp \ll{\alpha}{i}{}(x) \rp \mathbbm{1}_{A^L_i}(x).
      \end{split}
    \end{equation*}
    Their size, $L_r^r(f^U_{i,\alpha}, f^L_{i,\alpha})$ is bounded above by
    \begin{equation*}
      \sup_{y \in [y_1,y_0]} h'(y)^r \int_{A^L_{i}} \lp
      \uu{\alpha}{i}{}-  \ll{\alpha}{i}{} \rp^r d\lambda
      + h(y_0)^r \int_{A^U_{i} \setminus A^L_{i} } d\lambda
      \le M^r \epsilon^r
      + h(y_0)^r \epsilon^r
    \end{equation*}
    for some $0 < M < \infty$ by Assumption~\ref{assum:enum:uniflip}.  Thus
    the bracket size is of order $\epsilon$, as desired. The log cardinality
    $\log N^B N^S$ is
    \begin{equation*}
      C \lp \frac{|y_1|}{\epsilon}  \rp^{1/2}
      + 4 \log (\epsilon^{-r}).
    \end{equation*}
    Thus, we get the same conclusion as in the case $\tilde y_0 = -\infty$,
    and we have completed the proof for $\epsilon < \epsilon^*$.

    When either $\tilde y_0 = -\infty$ or $\tilde y_0 > -\infty$, we have
    proved the theorem when $0 < \epsilon \le \epsilon^*$.  The result can be
    extended to apply to any $\epsilon>0$ in a manner identical to the
    extension at the end of the proof of Proposition~\ref{prop:GSextension}.
  \end{proof}

  \begin{proof}[Proof of Proposition~\ref{prop:calPMOnehEnvelope}]

    First we find an envelope for the class $\mathcal{P}_{M,h}$ with $
    \alpha_h > 1$.  For $x \in [-(2M+1),2M+1]$, the envelope is trivial.
    Thus, let $x \ge 2M+1$.  The argument for $x \le -(2M+1)$ is symmetric.
    We show the envelope holds by considering two cases for $p =
    h\circ\varphi \in \Pmh$.  Let $R\equiv \dom \varphi \cap [1, \infty)$.
    First consider the case
    \begin{equation}
      \label{eq:case2}
      \inf_{x \in R} p(x) \le    1/(2M).
    \end{equation}
    We pick $x_1 \in R$ such that $p(x_1) = h(\varphi(x_1)) = 1/(2M)$ and
    such that
    \begin{equation}
      \label{eq:L-bound}
      \varphi(0) - \varphi(x_1) \ge   h^{-1}(M^{-1}) - h^{-1}(M^{-1}/2) \equiv L  > 0.
    \end{equation}
    This is possible since $\varphi(0) \ge h^{-1}(M^{-1})$ by the
    definition of $ {\cal P}_{1,M,h}$  and by  our choice of $x_1$
    (and by the fact that $\dom \varphi$ is closed, so that we attain equality
    in \eqref{eq:case2}).

    If $p(z) \ge 1/(2M)$, then concavity of $\vp$ means $p \ge 1/(2M)$ on
    $[0,z]$ and since $p$ integrates to $1$, we have $z \le 2M$.  Thus $x_1 \le
    2M$. Fix $x > 2M+1
    \ge x_1 > 0$, which (by concavity of $\varphi$) means $\varphi(0) >
    \varphi(x_1) > \varphi(x)$.  We will use
    Proposition~\ref{propGeneralHconcaveUpperBounds} with $x_0=0$ and $x_1$ and
    $x$ as just defined.  Also, assume $\varphi(x) > -\infty$, since otherwise
    any $0<D, L < \infty$ suffice for our bound.  Then, we can apply
    \eqref{eq:f-bound-1} to see
    \begin{equation}
      p(x) \le   h\left(
        \varphi(0) - h(\varphi(x_1)) \frac{\varphi(0) - \varphi(x_1)}{F(x)-F(0)} x
      \right).
      \label{eq:f-bound-2}
    \end{equation}
    Since $(F(x) - F(0))^{-1} \ge 1$ (since $\alpha > 1$),
    \eqref{eq:f-bound-2} is bounded above by
    \begin{equation}
      h\left(
        h^{-1}(M) - \frac{L}{2M }  x
      \right) < \infty.
      \label{eq:h-simple}
    \end{equation}
    We can assume $h^{-1}(M) = -1$ without loss of generality.  This is
    because, given an arbitrary $h$, we let $h_M(y) = h(y + 1 + h^{-1}(M))$
    which satisfies $h_M^{-1}(M)=-1$.  Note that $\Pmh = \mathcal{P}_{M,h_{M}}$
    since translating $h$ does not change the class $\mathcal{P}_h$ or $\Pmh$.
    Thus, if \eqref{eq:DefnOfCalPMOnehEnvelope} holds for all $p \in
    \mathcal{P}_{M,h_{M}}$ then it holds for all $p \in \Pmh$.  So without loss
    of generality, we assume $h^{-1}(M)=-1$. Then \eqref{eq:h-simple} is equal
    to
    \begin{equation}
      \label{eq:1}
      h\left(
        -1 - \frac{L}{2M }  x
      \right)
      <\infty.
    \end{equation}
    Now, $h(y) = o(|y|^{-\alpha})$ as $y \to -\infty$, which implies that $h(y) \le D
    (-y)^{-\alpha}$ on $(-\infty, -1]$ for a constant $D$ that depends only on
    $h$ and on $M$, %
    since $-1 - (L/(2M)) x \le -1 $.  Thus, \eqref{eq:h-simple} is bounded
    above by
    \begin{equation}
      \label{eq:final}
      D \left( 1 + \frac{L}{2M } x
      \right)^{-\alpha}.
    \end{equation}

    We have thus found an envelope for the case wherein \eqref{eq:case2}
    holds and when $x \ge 2M+1$.  The case $x \le -(2M+1)$ is symmetric.

    Now consider the case where $p$ satisfies
    \begin{equation}
      \label{eq:inf-f}
      \inf_{x \in R} p(x) \ge 1/(2M).
    \end{equation}
    As argued earlier, if $p(z) \ge 1/(2M)$, then concavity of $\vp$ means $p
    \ge 1/(2M)$ on $[0,z]$ and since $p$ integrates to $1$, we have $z \le
    2M$.  So, when \eqref{eq:inf-f} holds, it follows that $p(z) = 0$ for $z
    > 2M$.  We have thus shown 
    $p \le p_{u,h}$ (with $p_{u,h}$ defined in
    \eqref{eq:DefnOfCalPMOnehEnvelope}).  
    For $q \equiv p^{1/2} \in \mathcal{P}_{M,h}^{1/2}$, it is now immediate
    that $q \le p_{u,h}^{1/2}$.
  \end{proof}

  To prove Theorem~\ref{thm:entropy_noncompactx_RRd}, we partition $\RR$ into
  intervals, and on each interval we apply
  Theorem~\ref{thm:entropy_compactx}.  The envelope from
  Proposition~\ref{prop:calPMOnehEnvelope} gives a uniform bound on the
  heights of the functions in $\Pmh^{1/2}$, which allows us to control the
  cardinality of the brackets given by Theorem~\ref{thm:entropy_compactx}.

\begin{center}
  \begin{figure}
    \caption[Bracketing Construction]{
      Theorem~\ref{thm:entropy_compactx}: Bracketing of a concave function
      $\vp$ (rather than $h(\vp)$).  Here $\AL{i_\gamma}{\gamma} = [a_{l_1},
      a_{l_2}]$ and $\AU{i_\gamma}{\gamma} = [a_{l_1-1}, a_{l_2+1}]$, and the
      right boundary of the domain of $\vp $ lies between $a_{l_2}$ and
      $a_{l_2+1}$.  We focus on the right side, near $a_{l_2}$ and
      $a_{l_2+1}$.  In the top plot is a bracket on the domain
      $\cup_{j=1}^{\gamma-1} \AU{i_j}{j}$ (which we let have right endpoint
      $b$ here) and the range $[y_{\gamma-1},y_0]$ (below which $\vp$ is
      greyed out).  The next plot shows an application of
      Proposition~\ref{prop:GSextension} to find a bracket on
      $\AL{i_\gamma}{\gamma}$.  The final plot shows the combination of the
      two.  %
    }
    \includegraphics[scale=.55]{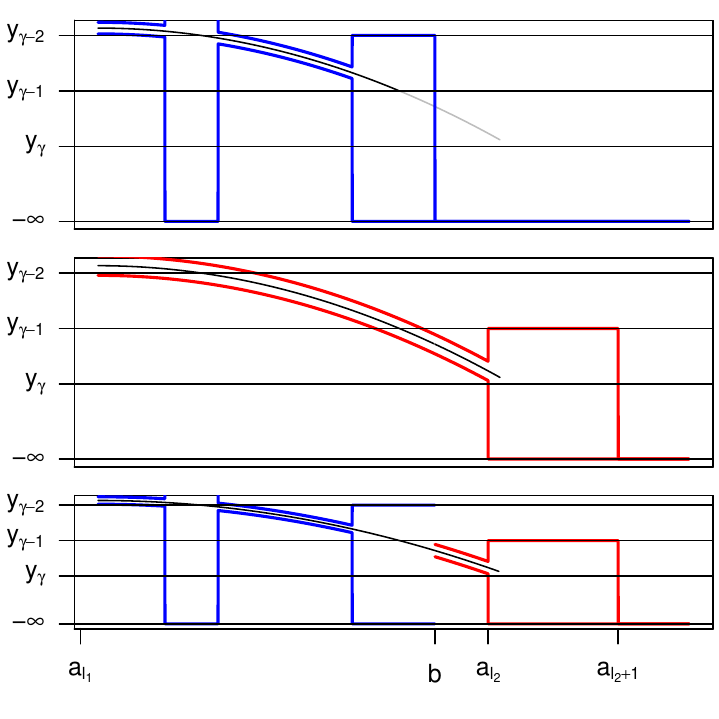}
    \label{fig:bracket-join}
  \end{figure}
  \end{center}

  \begin{proof}[Proof of Theorem~\ref{thm:entropy_noncompactx_RRd}]

    We will use the method of Corollary~2.7.4 of \cite{MR1385671} %
    for combining brackets on a partition of $\RR$, together with
    Theorem~\ref{thm:entropy_compactx}.  Let $I_0 = [-(2M+1),2M+1]$; for $ i
    > 0$ let $I_i = [i^\gamma, (i+1)^\gamma] \setminus I_0$, and for $i < 0 $
    let $I_i = [ -|i-1|^\gamma, -|i|^\gamma] \setminus I_0$.
    Let $A_0 = M^{1/2} (4M+2)^{1/r} $ and $A_i = D^{1/2} \lp 1 + |i|^\gamma
    L/(2M)\rp^{-\alpha} \lp (i+1)^\gamma - i^\gamma \rp^{1/r} $ where $\alpha
    \equiv \alpha_{h^{1/2}}$ (so by Lemma~\ref{lem:transform-basic-facts}
    $\alpha_h = 2\alpha_{h^{1/2}} > 1$) for $|i| > 0$, and with $D,L$ as
    defined in Proposition~\ref{prop:calPMOnehEnvelope}, which will
    correspond to $B(b_2-b_1)^{1/r}$ in Theorem~\ref{thm:entropy_compactx}
    for $\mathcal{P}^{1/2}_{M,h}$
    restricted to $I_i$.  For $i \in \ZZ$, let $a_i = A_i^{\beta} $
    where we will pick $\beta \in (0,1)$ later.
    Fix $\epsilon > 0$.  We will  apply Theorem~\ref{thm:entropy_compactx}
    to yield $L_r$ brackets of size $\epsilon a_{i}$ for the restriction of
    ${\cal P}_{M,h}^{1/2}$ to each interval $I_{i}$.
    %
    %
    %
    %
    %
    For $i \in \ZZ$
    we apply
    Theorem~\ref{thm:entropy_compactx} and form $\epsilon a_i$ brackets,
    which we denote by $[f^L_{i, j}, f^U_{i,j} ]$ for $j=1,\ldots, \NB{i}$,
    for the restriction of ${\cal P}_{M,h}^{1/2}$ to $I_i$. We will bound
    $\NB{i}$ later.
    %
    %
    %
    %
    %
    We have thus formed a collection of brackets for ${\cal P}_{M,h}^{1/2}$
    by
    \begin{equation*}
      \lb
      \ls \sum_{i \in \ZZ} f_{i,j_i}^L \mathbbm{1}_{I_i} ,
      \sum_{i \in \ZZ} f_{i,j_i}^U \mathbbm{1}_{I_i}
      \rs
      : j_i \in \lb 1, \ldots, N_i\rb, i \in \ZZ
      \rb.
    \end{equation*}
    The cardinality of this bracketing set is $\prod_{i \in \ZZ} \NB{i}$.
    The $L_r^r$ size of a bracket $[f^L,f^U]$ in the above-defined collection
    is
    \begin{equation*}
      \int_{\RR} | f^U-f^L|^r d\lambda
      \le
      \sum_{i \in \ZZ}
      \epsilon^r  a_{i}^r.
    \end{equation*}
    By Theorem~\ref{thm:entropy_compactx}, $\log \NB{i} \le
    \tilde K_{r,h} (A_{i} / (\epsilon a_{i}))^{1/2}$ for $i \in \ZZ$ where
    $    \tilde K_{r,h}$ is the constant
    from that theorem.  Thus,
    \begin{equation*}
      \log N_{[\, ]} \lp \epsilon \lp \sum_{i \in \ZZ}
      a_{i}^r\rp^{1/r},
      {\cal P}_{M,h}^{1/2}
      , L_r
      \rp
      \le
      \tilde K_{r,h} \sum_{i \in \ZZ} \lp \frac{A_i }{\epsilon a_i} \rp^{1/2}.
    \end{equation*}
    We now set $\beta = 1/(2r+1)$, so that $ a_i^r = \lp A_{i} / a_{i}
    \rp^{1/2} = A_i^{r / (2r+1)}$ and need only to compute $ \sum_{i \in \ZZ}
    a_i^r = \sum_{i \in \ZZ} \lp A_i / a_i \rp^{1/2}.$ Let $\tilde A_i = A_i
    / D^{1/2}$, and we then see that
    \begin{align*}
      \sum_{|i| \ge 1} \tilde A_i^{r/(2r+1)} & = 2 \sum_{i \ge 1} \lp 1 +
      \frac{L}{2M} i^\gamma \rp^{- \alpha r / (2r+1)}
      \lp  (i+1)^\gamma - i^\gamma \rp^{1/(2r+1)} \\
      & \le 2 \sum_{i \ge 1} \lp 1 + \frac{L}{2M} i^\gamma \rp^{-\alpha r /
        (2r+1)}
      i^{\gamma/(2r+1)} \lp  \lp \frac{i+1}{i}\rp^\gamma - 1 \rp^{\inv{(2r+1)}} \\
      & = 2^{1 + \gamma/(2r+1)} \sum_{i \ge 1} \lp 1 + \frac{L}{2M} i^\gamma
      \rp^{-\alpha r / (2r+1)} i^{\gamma/(2r+1)} \\
      & \le 2^{1 + \gamma/(2r+1)} \sum_{i \ge 1} \lp \frac{L}{2M} i^\gamma
      \rp^{-\alpha r / (2r+1)} i^{\gamma/(2r+1)}
    \end{align*}
    which equals
    \begin{align*}
      \MoveEqLeft 2^{1 + \gamma/(2r+1)} \lp \frac{L}{2M} \rp^{-\alpha
        r/(2r+1)}
      \sum_{i \ge 1} i^{- \gamma
        \alpha r / (2r+1) + \gamma/(2r+1)} \\
      & \le  2^{1 + \gamma/(2r+1)} \lp \frac{L}{2M} \rp^{-\alpha
        r/(2r+1)} \lp 1 + \int_1^\infty x^{-\alpha \gamma r /(2r+1) +
        \gamma/(2r+1)} \rp  dx%
    \end{align*}
    which equals
    \begin{equation}
      2^{1 + \gamma/(2r+1)} \lp \frac{L}{2M} \rp^{-\alpha
        r/(2r+1)} \lp 1 + \inv{\frac{\alpha \gamma r}{2r + 1}
        - \frac{\gamma}{2r+1} - 1 } \rp
      \label{eq:sum-exclude0-arbitrary-gamma}
    \end{equation}
    as long as
    \begin{equation*}
      \frac{\alpha \gamma r}{2r + 1}     - \frac{\gamma}{2r+1} > 1
    \end{equation*}
    which is equivalent to requiring
    \begin{equation}
      \label{eq:alpha-condition-arbitrary-gamma}
      \alpha > \inv{r} + \frac{2r+1}{r} \inv{\gamma}.
    \end{equation}
    Since $\gamma \ge 1$ is arbitrary, for any $\alpha > 1/r$, we can pick
    $\gamma = ((2r+1)/r)  2 / (\alpha-1/r)$.
    Then the right side of
    \eqref{eq:alpha-condition-arbitrary-gamma} becomes $(1/r) ( 1 - 1/(2r)) +
    \alpha / (2r)$, and thus \eqref{eq:alpha-condition-arbitrary-gamma}
    becomes
    \begin{equation*}
      \alpha > \frac{\alpha + \inv{r}}{2},
    \end{equation*}
    which is satisfied for any $r \ge
    1$ and $\alpha > 1/r$.  Then \eqref{eq:sum-exclude0-arbitrary-gamma} equals
    \begin{equation*}
      2^{2 +  \frac{2}{\alpha-1/r} \inv{r}} \lp \frac{L}{2M} \rp^{ -\alpha
        r/(2r+1)}.
    \end{equation*}
    Thus, defining $K_{r,\alpha} \equiv \sum_{i \in \ZZ}  A_i^{r/(2r+1)}$, we have
    \begin{align*}
      K_{r,\alpha} & = M^{r/(2(2r+1))} (4M+2)^{1/(2r+1)}
      +
      D^{r/(2(2r+1))}2^{2 + \frac{2}{\alpha-1/r} \inv{r}} \lp \frac{L}{2M} \rp^{-\alpha
        r/(2r+1)}.
    \end{align*}
    Then we have shown that
    \begin{equation*}
      \log N_{[\, ]} \lp \epsilon K_{r,\alpha}^{1/r},
      {\cal P}_{M,h}^{1/2}
      , L_r
      \rp
      \le
      \tilde K_{r,h}
      K_{r,\alpha} \epsilon^{-1/2},
    \end{equation*}
    or
    \begin{equation*}
      \log N_{[\, ]} \lp \varepsilon,
      {\cal P}_{M,h}^{1/2}
      , L_r
      \rp
      \le
      \tilde K_{r,h}      
      K_{r,\alpha}^{1 + 1/(2r)}  \varepsilon^{-1/2},
    \end{equation*}
    and the proof is complete.
  \end{proof}

  \begin{proof}[Proof of Theorem~\ref{thm:HellingerRateTheoremFinal}]

  {\bf Step 1: Reduction from ${\cal P}_h$ to ${\cal P}_{M,h}$.}  We first show that
    we may assume, without loss of generality, for some $M > 0$ that 
    $p_0 \in {\cal P}_{M,h}$
    and, furthermore, $\widehat{p}_n \in {\cal P}_{M,h}$ with probability
    approaching $1$ as $n \to \infty$.  To see this, consider translating and
    rescaling the data: we let $\tilde X_i = (X_i - b)/a$ for $b \in \RR$
    and $a > 0$, so that the $\tilde X_i$ are i.i.d.\ with density 
    $\tilde{p}_0(x) = a p_0( ax +b) $.  Now the MLE of the rescaled data,
    $\widehat p_n( \tilde x ; \underline{\tilde X})$ satisfies 
    $\widehat{p}_n(\tilde x; \underline{\tilde X}) = a \widehat p_n( a\tilde x +b); \underline X )$ 
    and, since the Hellinger metric is invariant under
    affine transformations, it follows that
    \begin{equation*}
      H\lp \widehat p_n(\cdot; \underline X), p_0 \rp
      = H\lp \widehat p_n(\cdot ; \underline{\tilde X} ), \tilde p_0 \rp.
    \end{equation*}
    Hence if \eqref{eq:hellinger_stdrate_anysupport_UC_MC} holds for
    $\tilde{p}_0$ and the transformed data, it also holds for $p_0$ and the
    original data.  Thus, we can pick $b$ and $a$ as we wish.  First, we note
    that there is some
    interval $B(x_0, \delta) \equiv \{ z: \vert z - x_0 \vert \le  \delta \}$
    contained in the interior of  the support of $p_0 \in \mathcal{P}_h$ since $p_0$
    has integral $1$.  We take $b$ and $a$ to be $x_0$ and $ \delta$,
    and thus assume without
    loss of generality that $B(0,1)$ is in the interior of the support of $p_0$.
    Now, by Theorem~2.17 of \cite{MR2766867} %
    which holds under their assumptions (D.1)--(D.4)
    it follows that we have uniform convergence of $\widehat{p}_n $ to $p_0$
    on compact subsets strictly contained in the support of $p_0$, such as
    $B(0,1)$.  Additionally, by Lemma~3.17 of
    \cite{MR2766867}, %
    we know that $\limsup_{n \to \infty} \sup_x \widehat p_n(x) \le \sup_x p_0 (x) \equiv M_0$
    almost surely.  The assumptions (D.1)--(D.4) of
    \cite{MR2766867} %
    for $h$ are implied by our
    \eqref{assum:enum:y0-infinite}--\eqref{assum:enum:yinfty-infinite} for $g
    \equiv h^{1/2}$ (with $\beta_h = 2 \beta_g$ and $\alpha_h = 2 \alpha_g$,
    since $h'(y) = 2 \sqrt{h(y)} (h^{1/2})'(y)$ and if $g'(y) =
    o(|y|)^{-(\alpha+1)}$ then $g(y) = o(|y|)^{-\alpha}$ as $y \to -\infty$).  
    Thus, we let 
    $M = \lp 1 + M_0 \rp \vee 2/\lp  \min_{ \vert x \vert \le 1}  p_0(x) \rp < \infty$.  
    Then we can henceforth assume that $p_0 \in {\cal P}_{M,h}$
    and, furthermore, with probability approaching $1$ as $n \to \infty$,
    that $\widehat p_n \in {\cal P}_{M,h}$.  This completes step $1$.

    \smallskip

    \par\noindent
    {\bf Step 2. Control of Hellinger bracketing entropy for $\Pmh$ suffices.}
    {\bf Step 2a:}  For
    $\delta> 0$, let
    \begin{equation*}
      \overline{\mathcal{P}}_{h}(\delta) \equiv
      \{ (p+p_0)/2 : \  p\in \mathcal{P}_{h}, \ H((p+p_0)/2 , p_0) < \delta\}.
    \end{equation*}
    Suppose that we can show that
    \begin{eqnarray}
      \log N_{[\, ]}(\epsilon, \overline{\mathcal{P}}_{h}(\delta), H)
      \lesssim \epsilon^{-1/2}
      \label{eq:P1bracketbound-d}
    \end{eqnarray}
    for all $0 < \delta \le \delta_0$ for some $\delta_0>0$.  Then it follows
    from \cite{MR1385671}, Theorems 3.4.1 and 3.4.4 (with $p_n=p_0$ in
    Theorem 3.4.4) or, alternatively, from \cite{MR1739079}, Theorem 7.4 and
    an inspection of her proofs, that any $r_n$ satisfying
    \begin{eqnarray}
      \label{eq:rateequation-d}
      r_n^2 \Psi ( 1/r_n) \le \sqrt{n}
    \end{eqnarray}
    where
    \begin{eqnarray*}
      \Psi (\delta) \equiv J_{[\, ]} (\delta , \overline{\cal P}_{h} (\delta ), H)
      \left (1 + \frac{J_{[\, ]} (\delta , \overline{\cal P}_{h} (\delta ), H)}{\delta^2 \sqrt{n}} \right )
    \end{eqnarray*}
    and
    \begin{eqnarray*}
      J_{[\, ]} (\delta , \overline{\cal P}_{h} (\delta ), H)
      \equiv \int_0^\delta \sqrt{\log N_{[\, ]} (\epsilon , \overline{\cal P}_{h} (\delta ), H)} d \epsilon
    \end{eqnarray*}
    gives a rate of convergence for $H(\widehat{p}_n , p_0)$.  It is easily seen that if
    \eqref{eq:P1bracketbound-d}
    holds
    then $r_n = n^{-2/5}$
    satisfies \eqref{eq:rateequation-d}.
    Thus
    \eqref{eq:hellinger_stdrate_anysupport_UC_MC} follows from
    \eqref{eq:P1bracketbound-d}.

    {\bf Step 2b.}  Thus we want to show that \eqref{eq:P1bracketbound-d}
    holds if we have an appropriate bracketing entropy bound for ${\cal
      P}_{M,h}^{1/2}$.  First note that
    $$
    N_{[\, ]} (\epsilon, \overline{\cal P}_{h} (\delta ), H)
    \le N_{[\, ]} (\epsilon, {\cal P}_{h} (4\delta ), H)
    $$
    in view of \cite{MR1385671}, exercise 3.4.4 (or \cite{MR1739079}), Lemma 4.2, page 48).
    Furthermore,
    $$
    N_{[\, ]} (\epsilon, {\cal P}_{h} (4\delta ), H) \le N_{[\, ]} (\epsilon, {\cal P}_{M,h}, H)
    $$
    since ${\cal P}_{h} (4\delta ) \subset {\cal P}_{M,h} $ for all $0 <
    \delta \le \delta_0$ with $\delta_0 >0$ sufficiently small.  This holds
    since Hellinger convergence implies pointwise convergence for concave
    transformed functions which in turn implies uniform convergence on
    compact subsets of the domain of $p_0$ via \cite{MR0274683}, Theorem
    10.8.  See Lemma~\ref{lemUniformOnCompacts} for details of the proofs.

    Finally, note that
    \begin{eqnarray*}
      N_{[\, ]} (\epsilon, {\cal P}_{M,h}, H)
      &= & N_{[\, ]} (\epsilon, {\cal P}_{M,h}^{1/2}, L_2( \lambda/2)) \\
      & = &  N_{[\, ]} (\epsilon, {\cal P}_{M,h}^{1/2}, L_2( \lambda)/\sqrt{2})
      = N_{[\, ]} (\epsilon/\sqrt{2}, {\cal P}_{M,h}^{1/2}, L_2( \lambda))
    \end{eqnarray*}
    by the definition of $H$ and $L_2 (\lambda)$.  Thus it suffices to show that
    \begin{equation}
      \label{eq:bracketFL2-d}
      \log N_{[\, ]}(\epsilon, \mathcal{P}_{M,h}^{1/2}, L_2(\lambda))
      \lesssim  \inv{\epsilon^{1/2}}
    \end{equation}
    where the constant involved depends only on $M$ and $h$.  This completes
    the proof of Step 2, and completes the proof, since
    \eqref{eq:bracketFL2-d} is exactly what we can conclude by
    Theorem~\ref{thm:entropy_noncompactx_RRd} since we assumed
    Assumption~\ref{assum:transformation} holds and that $\alpha \equiv
    \alpha_g$ satisfies $\alpha_g > 1/2$.
  \end{proof}

  \begin{proof}[Proof of Corollary~\ref{cor:SConcavePositiveConnectCor}]
    The proof is based on the proof of
    Theorem~\ref{thm:HellingerRateTheoremFinal}.  In Step 1 of that proof,
    the only requirement on $h$ is that we can conclude that
    $\widehat{p}_n$ is almost surely Hellinger consistent. Almost sure
    Hellinger consistency is given by Theorem~2.18 of
    \cite{MR2766867} %
    which holds under their assumptions (D.1)--(D.4), which are in turn
    implied by our \eqref{assum:enum:y0-infinite},
    \eqref{assum:enum:yinfty-finite}, and
    \eqref{assum:enum:yinfty-infinite} (recalling that all of our $h$'s are
    continuously differentiable on $(\tilde{y}_0, \tilde{y}_\infty)$).

    Then Step 2a of the proof shows that
    it suffices to show the bracketing bound
    \eqref{eq:P1bracketbound-d} for $\overline{{\cal P}}_{h} (\delta)$.
    Now, by
    Lemma~\ref{lem:convex-ordered-classes} below we have
    \begin{equation*}
      \log N_{[ \,]}(\epsilon, \overline{\mathcal{P}}_{h}(\delta), H ) \le
      \log N_{[ \,]}(\epsilon, \overline{\mathcal{P}}_{h_2}(\delta), H ).
    \end{equation*}
    Step 2b of the proof shows that  \eqref{eq:P1bracketbound-d}
    holds for transforms $h$
    when $g \equiv h^{1/2}$ satisfies
    $ \alpha \equiv \alpha_g > 1/2$, as we have assumed. Thus we are done.
  \end{proof}

  \section{Appendix: Technical Lemmas and Inequalities}
  \label{sec:Appendix1}

  We begin with the proof of Proposition~\ref{prop:GSextension}. It requires
  a result from \cite{Gunt-Sen:12}, so we will state that theorem, for the
  reader's ease.  The theorem gives bounds on bracketing numbers for classes
  of convex functions that are bounded and satisfy Lipschitz constraints.
  Let
  $\cC[ [a,b], [-B,B], \Gamma]$ be the class of functions  $ f \in \cC[
  [a,b], [-B,B] ]$
  satisfying the Lipschitz constraint $|f(x)-f(y)| \le \Gamma |x-y|$ for all
  $x,y \in [a,b]$.

  \begin{theorem}[Theorem 3.2 of \cite{Gunt-Sen:12}]
    \label{thm:GS3.2}
    There exist positive constants $c$ and $\epsilon_0$ such that for all
    $a<b$ and positive $B, \Gamma$, we have
    \begin{equation*}
      \log      \Nb[ \epsilon, {\cC[ [a,b], [-B,B], \Gamma]}, L_\infty ]
      \le c \lp \frac{B + \Gamma(b-a) }{\epsilon} \rp^{1/2}
    \end{equation*}
    for all $0 < \epsilon \le \epsilon_0 \{B + \Gamma(b-a) \}.$
  \end{theorem}
  \begin{proof}
    \cite{Gunt-Sen:12} prove this statement for metric covering numbers
    rather than bracketing covering numbers, but when using the supremum
    norm, the two are equal, if $\epsilon$ is adjusted by a factor of $2$: If
    $f_1, \ldots, f_N$ are the centers of $L_\infty$ balls of radius
    $\epsilon$ that cover a function class $\cal C$, then $[f_i - \epsilon,
    f_i + \epsilon]$, $i=1,\ldots,N$, are brackets of size $2 \epsilon$ that
    cover $\cal C$ (see e.g.\ page 157, the proof of Corollary~2.7.2, of
    \cite{MR1385671}). %
  \end{proof}

  \begin{proof}[Proof of Proposition~\ref{prop:GSextension}]
    First, notice that the $L_r$ bracketing numbers scale in the following
    fashion.  For a function
    $f \in \mathcal{C}([b_1,b_2], [-B,B])$
    we can define
    \begin{equation*}
      \tilde{f}(x) := \frac{ f(b_1 + (b_2-b_1)x)
        - B}{ B},
    \end{equation*}
    a scaled and translated version of $f$ that satisfies $\tilde{f} \in
    \mathcal{C}([0,1], [-1,1])$.  Thus, if $[l,u]$ is a bracket for $
    \mathcal{C}([b_1,b_2], [-B,B])$, then we have
    \begin{equation*}
      B^r  \int_0^1 \left\vert  \tilde{u}(x)-\tilde{l}(x) \right\vert^r dx
      = \inv{b_2-b_1} \int_{b_1}^{b_2} \left\vert u(x)-l(x) \right\vert^r dx.
    \end{equation*}
    Thus an $\epsilon-$size $L_r$ bracket for $\mathcal{C}([0,1], [-1,1])$
    immediately scales to be an $\epsilon (b_2-b_1)^{1/r}B $ bracket for
    $\mathcal{C}([b_1,b_2], [-B,B])$.
    Thus for the remainder of the proof we set
    $b_1=0$, $b_2=1$, and $B=1$.

    We take the domain to be fixed for these classes so that we can apply
    Theorem 3.2 of \cite{Gunt-Sen:12} which is the building block of the
    proof.  Now we fix
    \begin{equation}
      \label{eq:defn_u_v}
      \mu:= \exp(-2(r+1)^2(r+2) \log 2)
      \;\;\; \mbox{ and } \;\;
      \nu := 1-\mu.
    \end{equation}
    (Note that $\mu$ and $\nu$ are $u$ and $v$, respectively, in
    \cite{Gunt-Sen:12}.)  We will consider the intervals $[0,\mu]$, $[\mu,
    \nu]$, and $[\nu,1]$ separately, and will show the bound
    \eqref{eq:extensionGSbound} separately for the restriction of
    $\mathcal{C}([0,1],[-1,1])$ to each of these sub-intervals. This will
    imply \eqref{eq:extensionGSbound}.  We fix $\epsilon > 0 $, let $\eta
    =(3/17)^{1/r} \epsilon $, choose an integer $A$ and $\delta_0,\ldots,
    \delta_{A+1}$ such that
    \begin{equation}
      0=\delta_0 < \eta^r = \delta_1 <   \delta_2 < \cdots
      < \delta_A < \mu \le \delta_{A+1}.
    \end{equation}
    For two functions $f$ and $g$ on $[0,1]$, we can decompose the integral
    $\int_0^1 |f-g|^r d\lambda$ as
    \begin{equation}
      \int_0^1 |f-g|^r d\lambda =
      \int_0^\mu |f-g|^r d\lambda + \int_\mu^\nu |f-g|^r d\lambda + \int_\nu^1 |f-g|^r d\lambda.
    \end{equation}
    The first term and last term are symmetric, so we consider just the first
    term, which can be bounded by
    \begin{equation}
      \int_0^\mu |f-g|^r d\lambda
      \le \sum_{m=0}^A \int_{\delta_m}^{\delta_{m+1}} \vert f - g \vert^r d\lambda,
    \end{equation}
    since $\delta_{A+1} \ge \mu$. Now for a fixed $m \in \{1,\ldots,A\}$, we
    consider the problem of covering the functions in
    $\mathcal{C}([0,1],[-1,1])$ on the interval $[\delta_m,\delta_{m+1}]$.
    Defining $\tilde{f}(x) = f(\delta_m +(\delta_{m+1}-\delta_m) x)$ and
    $\tilde{g}(x) = g(\delta_m +(\delta_{m+1}-\delta_m) x)$, we have
    \begin{equation}
      \int_{\delta_m}^{\delta_{m+1}} \vert f-g \vert^r d\lambda
      = (\delta_{m+1}-\delta_m) \int_0^1 \vert \tilde{f}-\tilde{g} \vert^r d\lambda.
    \end{equation}
    Since concavity is certainly preserved by restriction of a function, the
    restriction of any function $f$ in $\mathcal{C}([0,1],[-1,1])$ to
    $[\delta_m, \delta_{m+1}]$ belongs to the Lipschitz class
    $\mathcal{C}([\delta_m, \delta_{m+1}], [-1,1], 2/\delta_m)$ (since $f$
    cannot ``rise'' by more than $2$ over a ``run'' bounded by $\delta_m$).
    Thus the corresponding $\tilde{f}$ belongs to $\mathcal{C}([0,1], [-1,1],
    2 (\delta_{m+1}-\delta_m)/ \delta_m)$.  We can now use
    Theorem~\ref{thm:GS3.2} to assert the existence of positive constants
    $\epsilon_0$ and $c$ that depend only on $r$ such that for all $\alpha_m
    \le \epsilon_0$ there exists an $\alpha_m$-bracket for $\mathcal{C}(
    [0,1], [-1,1], 2 (\delta_{m+1}-\delta_m)/ \delta_m)$ in the supremum norm
    of cardinality smaller than
    \begin{equation}
      \exp\left( c\alpha_m^{-1/2} \left( 2 + \frac{2 (\delta_{m+1}-\delta_m)}{\delta_m}
        \right)^{1/2}
      \right)
      \le \exp \left( c \left(\frac{\delta_{m+1}}{\delta_m \alpha_m} \right)^{1/2}
      \right).
      \label{eq:covercardinality}
    \end{equation}
    Denote the brackets by $\{ [l_{m,n_m}, u_{m,n_m}]: \ n_m = 1,\ldots,
    N_m\}$ where $N_m$ is bounded by \eqref{eq:covercardinality} and $m = 1,
    \ldots, A$.  Now, define the brackets $[ l_{n_m}, u_{n_m} ]$ by
    \begin{equation}
      \label{eq:defnbrackets}
      \begin{array}{l}
        l_{n_m} (x) \equiv
        -1_{[0,\delta_1]}(x) +\sum_{m=1}^A 1_{[\delta_m,\delta_{m+1}]}(x)  \,
        l_{m,n_m}(x), \\
        u_{n_m} (x) \equiv 1_{[0,\delta_1]}(x) +\sum_{m=1}^A
        1_{[\delta_m,\delta_{m+1}]}(x) \,  u_{m,n_m}(x)
      \end{array}
    \end{equation}
    for the restrictions of the functions in $\mathcal{C}([0,1],[-1,1])$
    to the set $[0,\mu]$, where the tuple $(n_1,\ldots,n_A)$ defining the
    bracket varies over all possible tuples with components $n_m \le N_m$,
    $m=1,\ldots,A$. The brackets were chosen in the supremum norm, so we
    can compute their $L_r(\lambda)$ size as $S_1^{1/r}$ where
    \begin{equation}
      \label{eq:S1}
      S_1 = \delta_1 + \sum_{m=1}^A  \alpha_m^r (\delta_{m+1}-\delta_m),
    \end{equation}
    and the cardinality is $\exp(S_2)$ where
    \begin{equation}
      \label{eq:S2}
      S_2 = c \sum_{m=1}^A \left( \frac{2 \delta_{m+1}}{\delta_m \alpha_m} \right)^{1/2}.
    \end{equation}
    Thus our $S_1$ and $S_2$ are identical to those in (7) in
    \cite{Gunt-Sen:12}. Thus, by using their choice of $\delta_m$ and
    $\alpha_m$,
    \begin{equation*}
      \begin{split}
        \delta_m = \exp\lp r \lp \frac{r+1}{r+2} \rp^{m-1} \log \eta \rp, \\
        \alpha_m = \eta \exp \lp -r  \frac{ (r+1)^{m-2}}{(r+2)^{m-1} } \log
        \eta \rp,
      \end{split}
    \end{equation*}
    their conclusion that
    \begin{equation*}
      S_1 \le \frac{7}{3} \eta^r \mbox{ and } S_2 \le 2c  \lp\frac{2}{\eta} \rp^{1/2}
    \end{equation*}
    holds.

    An identical conclusion holds for the restriction of $f \in \cC[[0,1],[-1,1]]$ to $[\nu,1]$.  Finally, if $f \in \cC[[0,1],[-1,1]]$
    then its restriction to $[\mu,\nu]$ lies in $\cC[ [\mu,\nu], [-B,B],
    2/\mu]$, for which, via Theorem~\ref{thm:GS3.2}, for all $\eta \le
    \epsilon_0$, we can find a bracketing of size $\eta$ in the $L_r$ metric
    (which is smaller than the $L_\infty$ metric) having cardinality smaller
    than
    \begin{equation*}
      \exp \lp c \eta^{-1/2}   \lp 2 + \frac{2}{\mu } \rp^{1/2} \rp
      \le \exp \lp c \lp  \frac{2}{\mu} \rp^{1/2} \lp \frac{2}{\eta} \rp^{1/2} \rp.
    \end{equation*}
    Thus we have brackets for $\cC[[0,1],[-1,1]]$ with $L_r$ size bounded by
    \begin{equation*}
      \lp \frac{7}{3}\eta^r + \frac{7}{3}\eta^r + \eta^r \rp^{1/r}
      = \lp \frac{17}{3} \rp^{1/r} \eta,
    \end{equation*}
    and log cardinality bounded by
    \begin{equation*}
      c \lp 4 + \lp\frac{2}{\mu}\rp^{1/2} \rp  \lp\frac{2}{\eta}\rp^{1/2} .
    \end{equation*}
    Since $\eta = (3/17)^{1/r} \epsilon$, we have shown that
    \begin{equation*}
      \log N_{[\, ]} (\epsilon, \mathcal{C}([0,1], [-1,1]), L_r)
      \le C_1 \lp\frac{1}{\epsilon} \rp^{1/2}
    \end{equation*}
    for a constant $C_1$ and $\epsilon \le \epsilon_3 \equiv (17/3)^{1/r}
    \epsilon_0$.

    To extend this result to all $\epsilon > 0$, we note that if $\epsilon
    \ge 2$, we can use the trivial bracket $[-1_{[0,1]}, 1_{[0,1]}]$.  Then,
    letting $C_2 = \frac{(1/\epsilon_3)^{1/2}}{1/2^{1/2}}$, for $\epsilon_3
    \le \epsilon \le
    2$ we have
    \begin{align*}
      C_2 \cdot C_1 \epsilon^{-1/2} \ge C_1 \epsilon_3^{-1/2}
      \ge       \log N_{[\, ]} (\epsilon, \mathcal{C}([0,1], [-1,1]), L_r),
    \end{align*}
    since bracketing numbers are non-increasing.  Thus, taking $C \equiv C_2
    \cdot C_1$, we have shown \eqref{eq:extensionGSbound} holds for all
    $\epsilon > 0$ with $[b_1,b_2]=[0,1]$ and $B=1$.  By the scaling
    argument at the beginning of the proof we are now done.
  \end{proof}

  For $\delta >0$ and $\mathcal{P}_h$ consisting of all $h$-concave densities
  on $\RR$ as in \eqref{eq:defn:Ph},
  let
  \begin{equation*}
    \mathcal{P}_h(\delta) \equiv \{ p \in \mathcal{P}_h : \  H(p, p_0) <
    \delta \},
  \end{equation*}
  \begin{equation*}
    \overline{\mathcal{P}}_h(\delta) \equiv
    \{ (p+p_0)/2
    : \ p \in \mathcal{P}_h, H((p+p_0)/2 , p_0) < \delta\},
  \end{equation*}
  and let $\Pmh$ be as defined in \eqref{eq:DefnOfCalPMOneH}.
  \begin{lemma}
    \label{lemUniformOnCompacts}
    Let $\delta>0$ and $0 < \epsilon \le \delta$.
    With the definitions in the previous display
    \begin{align}
      N_{[\, ]}(\epsilon, \overline{\mathcal{P}}_h(\delta), H)
      &  \lesssim N_{[\,]}(\epsilon, \mathcal{P}_h(4 \delta),H)     \label{eq:vandegeerClass_h_1} \\
      & < N_{[\,]}(\epsilon, \mathcal{P}_{M,h}, H). \label{eq:vandegeerClass_h_2}
    \end{align}
  \end{lemma}

  \begin{proof}
    We will follow the notation in \cite{MR1739079} %
    (see e.g.\ chapter 4) and set $\overline{p} = (p+ p_0) / 2$ for any
    function $p$.  Then if $\overline{p}_1 \in
    \overline{\mathcal{P}}_h(\delta)$, by (4.6) on page 48 of
    \cite{MR1739079}, we have $H(p_1,p_0) < 4 H(\overline{p}_1,p_0) <
    4\delta$, so that $p_1 \in \mathcal{P}_h(4\delta)$. Then given
    $\epsilon-$brackets $[l_\alpha,u_\alpha]$, of $\mathcal{P}_h(4\delta)$,
    with $1 \le \alpha \le N_{[\,]}(\epsilon, \mathcal{P}_h(4\delta), H)$, we
    can construct brackets of $\overline{\mathcal{P}}_h(\delta)$ since for any
    $p_1 \in \mathcal{P}_h(4\delta)$ which is bracketed by
    $[l_\alpha,u_\alpha]$ for some $\alpha$, $\overline{p}_1$ is bracketed by
    $[\overline{l}_\alpha,\overline{u}_\alpha]$, so that
    $[\overline{l}_\alpha,\overline{u}_\alpha]$ form a collection of brackets
    for $\overline{\mathcal{P}}_h(\delta)$ with size bounded by
    \begin{equation*}
      H(\overline{l}_\alpha,\overline{u}_\alpha) \le \frac{1}{\sqrt{2}} H(l_\alpha,u_\alpha)
      < \frac{1}{\sqrt{2}} \epsilon,
    \end{equation*}
    where we used (4.5) on page 48 of \cite{MR1739079}.  Thus we have a collection of
    brackets of Hellinger size $\epsilon / \sqrt{2} < \epsilon$ with cardinality bounded by
    $N_{[\,]}(\epsilon, \mathcal{P}_h(4 \delta),H)$ and
    \eqref{eq:vandegeerClass_h_1} holds.

    Next we show \eqref{eq:vandegeerClass_h_2}, which will follow from showing
    $\mathcal{P}_h(4 \delta) \subset \mathcal{P}_{M,h}$.  Now if
    $0 < M^{-1} < \inf_{x \in [-1,1]} p_0(x)$ then for any $p$  that
    has its mode in $[-1,1]$ and satisfies
    \begin{equation}
      \label{eq:hellingertouniformbound}
      \sup_{x \in [-1,1]} | p(x)- p_0(x) | \le \min\left( \inf_{x \in [-1,1]} p_0(x) -M^{-1},
        M- \sup_{x \in [-1,1]} p_0(x) \right),
    \end{equation}
    we can conclude that $ p \in \mathcal{P}_{M,h}$.

    The proof of Lemma 3.14 of \cite{MR2766867}
    shows that for any sequence of $h$-concave densities $p_i$,
    \begin{equation}
      \label{eq:hellingerimpliesunif}
      H(p_i,p_0) \to 0 \;\; \mbox{ implies } \;\; \sup_{x \in [-1,1]} | p_i (x) - p_0(x)| \to  0.
    \end{equation}
    This says that the topology defined by the Hellinger metric has more open
    sets than that defined by the supremum distance on $[-1,1]$, which implies
    that open supremum balls are nested within open Hellinger balls, i.e.\ for
    $\epsilon >0$
    \begin{equation}
      \label{eq:openballs}
      B_{\epsilon}(p_0, \sup_{[-1,1]})
      \subseteq B_{4 \delta}(p_0,H)
    \end{equation}
    for some $\delta > 0$, where $B_\epsilon(p_0,d)$ denotes an open ball
    about $p_0$ of size $\epsilon$ in the metric $d$.

    Now, if $p$ is uniformly within $\epsilon$ of $p_0$ on $[-1,1]$, then for
    $\epsilon$ small enough we know that the mode of $p$ is in $[-1,1]$.
    Thus for $0 < M^{-1} < \inf_{x \in [-1,1]} p_0(x)$ and $\delta $ small
    enough, any $p \in \mathcal{P}_h(4\delta)$ is also in $\mathcal{P}_{M,h}$
    as desired, and so \eqref{eq:vandegeerClass_h_2} has been shown.
  \end{proof}

  \begin{lemma}
    \label{lem:h-inverse}
    For a concave-function transformation $h$ that satisfies
    Assumption~\ref{assum:enum:y0-infinite}, we can have that $h^{-1}$ is
    nondecreasing and as $f \searrow 0$,
    \begin{equation}
      \label{eq:littleo-alpha}
      h^{-1}(f) = o(f^{-1/\alpha}).
    \end{equation}
    In particular, for $ f \in (0,L]$, $h^{-1}(f) \le M_{L} f^{-1/\alpha}$.
  \end{lemma}
  \begin{proof}
    Let $\im h = h(\dom h)$.
    For two increasing functions $h \le g$ defined on $(-\infty,\infty)$ taking
    values in $[-\infty,\infty]$, where $\im h$ and $\im g$ are both intervals,
    we will show that $g^{-1}(f) \le h^{-1}(f) $ for any $f \in \im h \cap \im
    g$. By definition, for such $f$, we can find a $z \in (-\infty,\infty)$
    such that $f = g(z)$.  That is, $g(z) = h(h^{-1})(f) \le g(h^{-1}(f))$
    since $h \le g$. Applying $g^{-1}$, we see $z = g^{-1}(f) \le h^{-1}(f)$,
    as desired.

    Then \eqref{eq:littleo-alpha} follows by letting $g(y) = \delta
    (-y)^{-\alpha}$, which has $g^{-1}(f) = - (\inv{\delta} f)^{-1/\alpha}$.
    The statement that $h^{-1}(f) \le M_L f^{-1/\alpha}$ follows since on
    neighborhoods away from $0$,
    $h^{-1}$ is bounded above and $f\mapsto f^{-1/\alpha}$ is bounded below.

    To see that $h^{-1}$ is nondecreasing, we differentiate to see
    $(h^{-1})'(f) = 1/h'(h^{-1}(f))$. Since $h' \ge 0$ so is $(h^{-1})'$.
  \end{proof}

  \label{subsec:sConcaveDensityBounds}

  \begin{proposition}
    \label{propGeneralHconcaveUpperBounds}
    Let $h$ be a concave-function transformation and $f= h \circ \varphi $
    for $\varphi \in {\cal C}$ and let $F (x) = \int_{-\infty}^x f(y)\,
    dy$. Then for $x_0 < x_1 < x$ or $x < x_1 < x_0$, all such that $- \infty
    < \varphi(x) < \varphi(x_1) < \varphi(x_0) < \infty$, we have
    \begin{equation}
      f(x) \le
      h\left(
        \varphi(x_0) - h(\varphi(x_1)) \frac{\varphi(x_0) - \varphi(x_1)}{F(x)-F(x_0)} (x-x_0)
      \right).   \label{eq:f-bound-1}
    \end{equation}
  \end{proposition}

  \begin{proof} %
    Take  $x_1 , x_2 \in \RR$ with $x_1 < x_2$. Then
    \begin{align*}
      F(x_2)-F(x_1)
      & = \int_{x_1}^{x_2} f(x) \, dx
      = \int_{x_1}^{x_2} h(\varphi(x)) \, dx \\
      & = \int_{x_1}^{x_2} h \left( \varphi \left(
          \frac{x_2-x}{x_2-x_1} x_1
          + \frac{x-x_1}{x_2-x_1} x_2
        \right) \right) \, dx,
    \end{align*}
    and since $h$ is nondecreasing and $\varphi $ is concave, the above is not smaller than
    \begin{align*}
      \int_{x_1}^{x_2} h\left(
        \frac{x_2-x}{x_2-x_1} \varphi(x_1)
        + \frac{x-x_1}{x_2-x_1} \varphi(x_2)
      \right) \, dx,
    \end{align*}
    which, by the change of variables $u = (x-x_1) / (x_2-x_1)$, can be written as
    \begin{equation}
      \label{eq:gdf-concavity-bound}
      \int_0^1 h\left(
        (1-u) \varphi(x_1) + u \varphi(x_2)
      \right)
      (x_2-x_1) \, du.
    \end{equation}

    Now we let $x_1 = x_0 $ and $x_2 = x $ with $x_0 < x_1 < x$ as in the statement.
    Since $x_0$ and $x_1$ are in $\dom \varphi$,
    \begin{equation}
      \label{eq:defn:C}
      C\equiv  \int_0^1 h( (1-u)\varphi(x_0) + u \varphi(x_1)) \, du
    \end{equation}
    satisfies
    \begin{equation}
      \label{eq:C-bounds}
      0 < h(\varphi(x_1))  \le C \le h(\varphi(x_0)).
    \end{equation}
    Now, let $\eta = (\varphi(x_0)-\varphi(x_1)) / (\varphi(x_0) - \varphi(x))$, so that $\eta
    \in (0,1)$ by the assumption of the proposition.  Then
    \begin{eqnarray*}
      \lefteqn{ \int_0^1 h( (1-u) \varphi(x_0) + u \varphi(x)) \, du  }\\
      & = & \left( \int_0^{ \eta } + \int_{ \eta }^1 \right) h( (1-u)\varphi(x_0) + u \varphi(x)) \, du \\
      &   \ge & \int_0^{  \eta } h( (1-u)\varphi(x_0) + u \varphi(x)) \, du.
    \end{eqnarray*}
    Then by the substitution
    $v = u / \eta$,
    this is equal to
    \begin{equation}
      \label{eq:}
      \int_0^1 h \left(
        (1 - \eta v )\varphi(x_0)
        +   \eta v \varphi(x)  \right) \eta  \, dv.
    \end{equation}
    which is
    \begin{equation}
      \int_0^1 h \left(
        (1-v) \varphi(x_0) + v \varphi(x_1)
      \right)
      \frac{\varphi(x_0) - \varphi(x_1)}{\varphi(x_0) - \varphi(x)}
      \, dv,
      \label{eq:C-inv-vpx}
    \end{equation}
    by the construction of $\eta$, i.e.\ because
    \begin{align*}
      (1 - \eta v )\varphi(x_0)
      +   \eta v \varphi(x)
      & =   \left(1 - \frac{ \varphi(x_0) - \varphi(x_1)}{\varphi(x_0)-\varphi(x)}v \right)\varphi(x_0)
      +  \frac{ \varphi(x_0) - \varphi(x_1)}{\varphi(x_0)-\varphi(x)}v \varphi(x) \\
      & = v \varphi(x_0) + \frac{\varphi(x_0) - \varphi(x_1)}{ \varphi(x_0) - \varphi(x) } v (\varphi(x) - \varphi(x_0))\\
      & = v \varphi(x_0) - v (\varphi(x_0) - \varphi(x_1))\\
      & = (1-v) \varphi(x_0) + v \varphi(x_1).
    \end{align*}
    And, by definition of $C$, \eqref{eq:C-inv-vpx} equals
    $  C (\varphi(x_0) - \varphi(x_1))/(\varphi(x_0) - \varphi(x))$.
    This gives, by applying \eqref{eq:gdf-concavity-bound}, that
    \begin{align}
      F(x)-F(x_0) & \ge (x-x_0) \int_0^1 h( (1-u)\varphi(x_0) + u \varphi(x) )\, du \notag \\
      & \ge   (x-x_0) C \frac{\varphi(x_0) - \varphi(x_1)}{\varphi(x_0) - \varphi(x)}. \label{eq:final-gdf-ineq}
    \end{align}

    Now we rearrange the above display to get an inequality for $\varphi(x)$.
    From \eqref{eq:final-gdf-ineq}, we have
    \begin{equation*}
      \varphi(x) \le \varphi(x_0) - C \frac{\varphi(x_0) - \varphi(x_1)}{F(x)-F(x_0)} (x-x_0),
    \end{equation*}
    and, since $h$ is nondecreasing,
    \begin{align*}
      h(\varphi(x))
      &\le h\left(
        \varphi(x_0) - C \frac{\varphi(x_0) - \varphi(x_1)}{F(x)-F(x_0)} (x-x_0)
      \right) \notag \\
      & \le
      h\left(
        \varphi(x_0) - h(\varphi(x_1)) \frac{\varphi(x_0) - \varphi(x_1)}{F(x)-F(x_0)} (x-x_0)
      \right),
    \end{align*}
    by \eqref{eq:C-bounds}.  This proves the claim for $x_0 < x_1 < x$.  The proof
    for $x < x_1 < x_0$ is similar.
  \end{proof}

  \begin{lemma}
    \label{lem:transform-basic-facts}
    If $g \equiv h^{1/2}$ is a concave-function transformation satisfying
    $g'(y) = o(|y|^{-(\alpha_g + 1)})$ then $g(y) = o(|y|^{-\alpha_g})$,
    $h(y) = o(|y|^{-2\alpha_g})$, and $h'(y) = o(|y|^{-(2\alpha_g+1)})$ as $y
    \to -\infty$.
  \end{lemma}
  \begin{proof}
    Since for any $\delta > 0$ we can find $N>0$ where for $y < -N$, $g(x) =
    \int_{-\infty}^x g'(y)dy \le \delta \int_{-\infty}^x
    (-y)^{-(\alpha_g+1)}$, we conclude that $g(y) = o(|y|^{-\alpha_g})$.  It
    follows additionally that $h(y) = o(|y|^{-2\alpha_g})$.  Thus for $\delta
    > 0$ there exists $N$ such that for $y < -N$, $h^{-1/2}(y) \ge
    \delta^{-1/2} |y|^{\alpha}$, and so we have that
    \begin{equation*}
      \delta |y|^{-(\alpha_g +1)} \ge h^{-1/2}(y) h'(y) 
      \ge \delta^{-1/2} |y|^{\alpha_g} h'(y)
    \end{equation*}
    since $2 g'(y) = h^{-1/2}(y) h'(y)$, so that $\delta^{3/2}
    |y|^{-(2\alpha_g+1)} \ge h'(y)$, as desired.
  \end{proof}

  \section*{Acknowledgements}
  We owe thanks to Arseni Seregin, Bodhi Sen, and Fadoua Balabdaoui as well
  as two referees for helpful comments, suggestions, and corrections.  Thanks
  also to Tilmann Gneiting for hospitality during our visits to Heidelberg.
  \bibliographystyle{imsart-number}
  \bibliography{HellingerRateNMR}

\end{document}